\documentclass[12pt]{article}
\usepackage{amsmath}
\usepackage{amssymb}
\usepackage{graphicx}
\usepackage{parskip}
\usepackage{amsthm}
\theoremstyle{definition}
\usepackage{listings}
\usepackage{color}
\usepackage{hyperref}
\usepackage{graphicx}
\usepackage{float}
\usepackage{caption}
\usepackage{multicol}
\usepackage[]{algorithm2e}

\usepackage[margin = 1in]{geometry}

\newtheorem{Thm}{Theorem}[section]

\newtheorem{Lem}{Lemma}[section]
\newtheorem{Def}{Definition}[section]
\newtheorem{Prop}{Proposition}[section]

\newtheorem{Conj}{Conjecture}[section]

\newcommand{\N}{\mathbb{N}}

\newcommand{\R}{\mathbb{R}}

\title{Isometries and Equivalences Between Point Configurations, Extended To $\varepsilon$-diffeomorphisms}
\author{Neophytos Charalambides\footnote{University of Michigan Math Department, 530 Church St, Ann Arbor, MI 48109, email: neochara@umich.edu}
\and S.B. Damelin\footnote{Department of Mathematics, University of Michigan, 530 Church Street, Ann Arbor, MI, USA., email: damelin@umich.edu}\and Brad Schwartz
\footnote{University of Michigan, 500 S State St, Ann Arbor, MI 48109, email: baschwa@umich.edu}}

\date{}

\begin{document}

\maketitle

\begin{abstract}
  \par{This announcement considers the following problem.  We deal with the \textit{Orthogonal Procrustes Problem}, in which two point configurations are compared in order to construct
    a map to optimally align the two sets. This  extends this to $\varepsilon$-diffeomorphisms, introduced by [1] Damelin and Fefferman. Examples
    will be given for when complete maps can not be constructed, for if the distributions do match, and finally an algorithm for partitioning
    the configurations into polygons for convenient construction of the maps.  A revision of this announcement is in the memoir preprint: arxiv: 2103.09748 [0], submitted for consideration for publication
\medskip} 
\end{abstract}


\tableofcontents

\section{Introduction}

\section{Procrustes Problem}

\textbf{Procrustes Problem:} \textit{Let $D,n \in \N$, $O(D)$ the set of orthogonal transformations, and $A(D)$ the set of affine transformations in $\R^{D}$. Considering the following action:
  $$x \mapsto Ax + \vec{t}$$
  where $A \in O(D)$, $\vec{t} \in \R^{D}$, $d(\cdot)$ the Euclidean distance in $\R^{D}$. $X = \{x_{1},\ldots,x_{n} \}$ and $Y = \{y_{1},\ldots,y_{n} \}$ are collections of distinct points in $\R^{D}$. If $d(x_i,x_i)=d(y_i,y_j)$ for all $1 \leq i,j \leq n$, then $\exists \varphi \in A(D)$ such that $\varphi(x_i)=y_i$ for all $i,j \in \{1,...,n\}$.} Given the collections $X$ and $Y$, find $A\in O(D)$ and $\vec{t}\in \R^D$, or $\phi\in A(D)$.
\vspace{2mm}

There are several proofs of this problem, which will be omitted.


\section{Non-Reconstructible Configurations}

In [2] we see an example of two set of points, where the set of their distances match but the point configurations do not, for a total number of points $n=4$. [2] Proposition 2.1 gives a simple way to check such cases.\\

\begin{Def}
  \textit{By relabelling, we mean that if there's an initial labelling of ordered points in two incongruent configurations, we reorder them in such a way that there's a correspondence between the points}.
\end{Def}

An example is if there's an initial labelling  of ordered points $\{a,b,c,d,e\}$ in $X$ and $\{\alpha,\beta,\epsilon,\delta,\gamma\}$ in $Y$, where $a$ corresponds to $\alpha$, $b$ to $\beta$, $c$ to $\gamma$, $d$ to $\delta$ and $e$ to $\epsilon$, one such relabelling will come from the permutation $\bigl(\begin{smallmatrix} 1 & 2 & 3 & 4 & 5\\ 1 & 2 & 5 & 4 & 3 \end{smallmatrix}\bigr)$.\\

\begin{Prop}
  $[2]$ \textit{Suppose } $n \neq 4.$ \textit{A permutation} $\phi \in S_{\binom{n}{2}}$ \textit{is a relabelling if and only if for all pairwise distinct indices} $i,j,k \in \{1,...,n\}$ \textit{we have:}
  $$\phi\cdot\{i,j\} \cap \phi\cdot\{i,k\} \neq \emptyset.$$
\end{Prop}

In other words, we need to take the edges of equal length between the two configurations we are considering and check if there's a mutual vertex between all such pairs for a given permutation $\phi \in S_{\binom{n}{2}}$. This permutation is what will give us the labelling if it does exist.\\

In the context of our problem, we consider the given n-point configurations $\{p_1,...,p_n\}$ and $\{q_1,...,q_n\}$ with their corresponding pairwise distances $D_P = \{dp_{ij} | dp_{ij}=d(p_i,p_j), 1\leq i,j\leq n\}$ and $D_Q = \{dq_{ij} | dq_{ij}=d(q_i,q_j), 1\leq i,j\leq n\}$ with $D_P = D_Q$ up to some reordering and $|D_P|=|D_Q|=\binom{n}{2}$.\\
\\
We then want to find if $\exists \{i,k\},\{j,l\}$ such that $d(p_i,p_k) = d(q_j,q_l) \Leftrightarrow dp_{ik} = dq_{jl} \text{   }\forall i,j,k,l \in \{1,...,n\}$ for a permutation $\phi \in S_{\binom{n}{2}}$. In the case where this isn't true, we need to disregard a certain number of $\textit{bad points}$ from both configurations in order to achieve this.\\

\subsection{Example}
Below is an example with two different 4-point configurations in $\R^2$ which have the same \textit{distribution of distances}. The corresponding equal distances between the 2 configurations are represented in the same color, and the we have two edges with distances $1, 2 \text{ and } \sqrt{5}$, but it's obvious that there doesn't exist a Euclidean transformation between the two.

From this example we can construct infinitely many sets of 2 different configurations with the same distribution of distances. This can be done by simply adding as many points as desired on the same location across the dashed line in both the configurations of figure 2.


\begin{figure}[h]
  \centerline{\includegraphics[width=15cm]{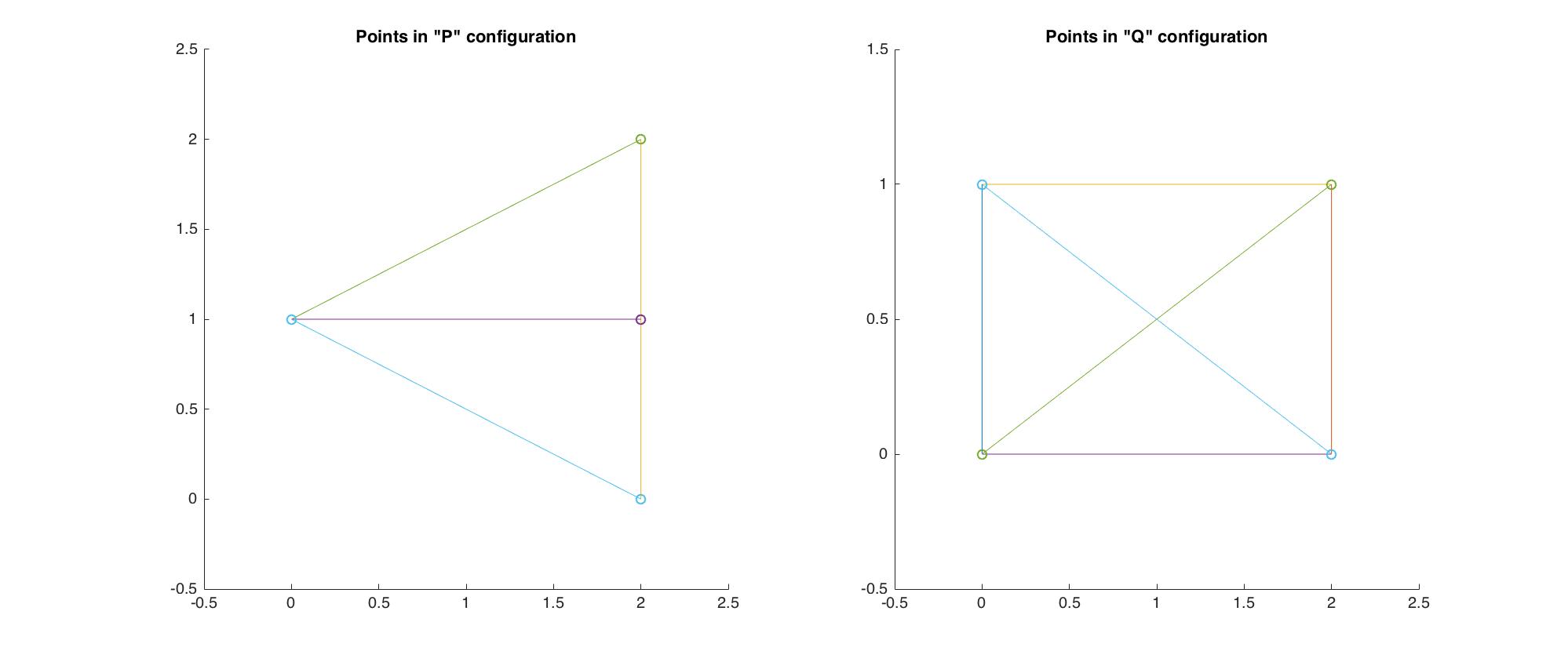}}
  \caption{Two different 4-point configurations with the same distribution of distances}
\end{figure}
\begin{figure}[h]
  \centerline{\includegraphics[width=15cm]{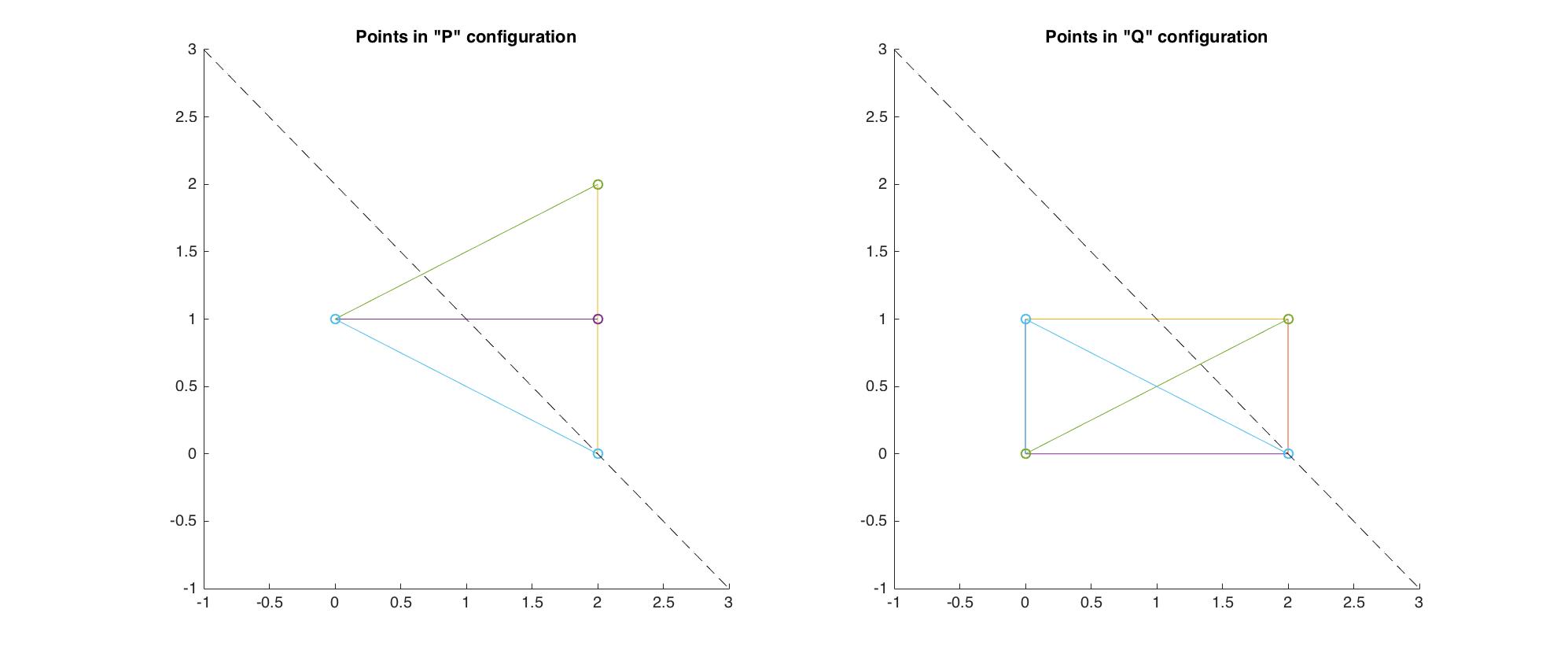}}
  \caption{Configurations with the same distribution of distances for $n\geq4$}
\end{figure}

In both the above example and the one mentioned in [2], it suffices to exclude one point from the two configurations and you will be able to get a Euclidean motion to move from one configuration to the other.
\vspace{5mm}

\begin{Conj}
  \textit{For two n-point configurations $P,Q\in\R^2$ with $D_P=D_Q$ for which $\nexists A\in O(2)$, $ \vec{t}\in\R^2$ such that $Q=\big(AP+\vec{t}\big)$ assuming the points have been labelled appropriately, then $\exists p_i\in P$ and $\exists q_i \in Q$ such that $Q\backslash q_i=\big(A(P\backslash p_i)+\vec{t}\big)$, for some $A\in O(2)$, $\vec{t}\in\R^2$.}
\end{Conj}
\vspace{2mm}

If the above conjecture holds, it suffices to exclude a single bad-point from both $P$ and $Q$, such that $P$ and $Q$ differ only by a Euclidean motion. Iterating through the potential pairs of bad-points will take $\mathcal{O}(n^2)$, the issue still arises in determining whether the points we excluded results in two congruent configurations.


\section{Partition Into Polygons}

One approach we can take in order to see which points should be excluded from our 2 configurations $P$ and $Q$, is to partition the entire configurations into smaller polygons and compare polygons of the \textit{same area}, in order to determine existing point correspondences between $P$ and $Q$. For any subsets $\{i,j,...\}\subseteq\{1,...,n\}$ or $\{s,t\}\subseteq\{1,...,k\}$ we consider in the upcoming sections, the elements of each subset will be distinct.

\subsection{Considering Areas Of Triangles - \textit{10-step algorithm}}
Considering our two n-point configurations $P=\{p_1,...,p_n\}$ and $Q=\{q_1,...,q_n\}$, we partition them into a total of $\binom{n}{3} \textit{ triangles }$ and considering the distance between our 3 points in each case, let's say indexed $i,j,k$, we have the distances $dp_{ij},dp_{ik},dp_{jk}$ and analogously $dq_{i'j'},dq_{i'k'},dq_{j'k'}$.\\

We now compute the areas as follows:
$$A_{ijk} = \sqrt{\frac{s(\frac{s}{2}-dp_{ij})(\frac{s}{2}-dp_{ik})(\frac{s}{2}-dp_{jk})}{2}} \qquad \text{where } s:=dp_{ij}+dp_{ik}+dp_{jk}$$
$$B_{i'j'k'} = \sqrt{\frac{s'(\frac{s'}{2}-dq_{i'j'})(\frac{s'}{2}-dq_{i'k'})(\frac{s'}{2}-dq_{j'k'})}{2}} \qquad \text{where } s':=dq_{i'j'}+dq_{i'k'}+dq_{j'k'}$$

and consider the sets of areas
$$\mathcal{A}=\big\{A_{ijk} |  \forall \{i,j,k\}\subseteq\{1,...,n\}\big\}$$
$$\mathcal{B}=\big\{B_{i'j'k'} |  \forall \{i',j',k'\}\subseteq\{1,...,n\}\big\}\}$$
where $|\mathcal{A}|=|\mathcal{B}|=\binom{n}{3}=\frac{n(n-1)(n-2)}{6}$.
\vspace{5mm}

We further partition the above sets as follows:
$$\mathcal{A}_1 = \big\{A_{ijk} | A_{ijk}\in \mathcal{A} \text{ and } \exists B_{i'j'k'} \in \mathcal{B} \text{ s.t. } A_{ijk}=B_{i'j'k'} \big\}$$
$$\mathcal{B}_1 = \big\{B_{i'j'k'} | B_{i'j'k'}\in \mathcal{B} \text{ and } \exists A_{ijk} \in \mathcal{A} \text{ s.t. } B_{i'j'k'}=A_{ijk} \big\}$$
$$\mathcal{A}_2 = \mathcal{A} \backslash \mathcal{A}_1 \qquad \qquad \mathcal{B}_2 = \mathcal{B} \backslash \mathcal{B}_1$$

Note that it may not be true that $|\mathcal{A}_1| = |\mathcal{B}_1|$, as the areas need not all be distinct.We essentially want the $\textit{shapes}$ formed by our points in the two sets which are ``$\textit{identical}$''. Let's assume our sets $\mathcal{A}_1$ and $\mathcal{B}_1$ are in ascending order with respect to the modes of the areas. We undertake the following steps in order to check which points to disregard and permutations are valid:
\begin{enumerate}
\item Disregard all points from from $P$ and $Q$ which are vertices of triangles in $\mathcal{A}_2$ and $\mathcal{B}_2$ respectively, but at the same time not vertices of any triangle in $\mathcal{A}_1$ and $\mathcal{B}_1$.
\item In order, we take $A_{ijk} \in \mathcal{A}_1$ and the corresponding triangles in $\mathcal{B}_1$, with $A_{ijk} = B_{i'j'k'}$.
\item If the distances of the sides of the triangles corresponding to $A_{ijk}$ and $B_{i'j'k'}$ don't match, disregard the triangles with area $B_{i'j'k'}$.
\item If the distances match up, we assign the points the corresponding points from $P$ to $Q$ and essentially start constructing our permutation, so thus far we have:
  $$\bigl(\begin{smallmatrix}
  i & j & k & \cdots \\
  i' & j' & k' & \cdots
\end{smallmatrix}\bigr)$$
  Alternatively, we can match the points between $A_{ijk}$ and $B_{i'j'k'}$ which have the same corresponding \textit{angles}.
\item Note that we might have more than 1 possible permutation, so for now we keep track of all of them and list them as $\alpha_s^{(t)}= \bigl(\begin{smallmatrix}
  i & j & k & \cdots \\
  i' & j' & k' & \cdots
\end{smallmatrix}\bigr)$ for $s$ being the indicator of the triangle we take from $\mathcal{A}_1$, and $t$ being the indicator of the corresponding triangle in $\mathcal{B}_1$ in order (so if 3 triangles correspond, we have $t \in \{1,2,3\}$).
  \begin{itemize}
  \item For triangles with 3 distinct inner angles we will have 1 permutation, for \textit{isosceles} triangles 2 permutations, and for \textit{equilateral} triangles  3!=6 permutations.
  \item In the case of \textit{squares} when considering quadrilaterals, we will have 4!=24 permutations (will be discussed in section 3.3).
  \item This can be thought of as \textit{matching angles} between equidistant edges of our polygons.
  \end{itemize}
\item Go to the next triangle in $\mathcal{A}_1$ (which might have the same area as our previous triangle), and repeat steps 2-4
  \begin{itemize}
  \item If the distances of our current triangle match with those of our previous triangle, simply take all previous permutations and ``$\textit{concatenate}$'' them. So for example $\big(\alpha_1^{(1)}\big)_{2} = \bigl(\begin{smallmatrix} \alpha_1^{(1)} \alpha_1^{(2)} \end{smallmatrix}\bigr) $, where the index $v$ in $\big(\alpha_s^{(t)}\big)_{v}$ indicates the combination we have with $\alpha_s^{(t)}$ being the first \textit{element} of the permutation as above. We therefore get a total of $\prod_{\iota=0}^{\nu-1}(t-\iota)$ permutations we are currently keeping track of, where $\nu$ is the number of \textit{elements} in the constructions thus far. Note that in the above procedure we assume no common points, and the case where mutual points exists is described below.
  \end{itemize}

\item If our current and previous triangles $\textit{share points}$, we consider the combination of two triangles in $\mathcal{B}_1$ with the same corresponding areas and shapes and matching points as the combination of the two triangles taken from $\mathcal{A}_1$, so we'll either get a quadrilateral (if they share 2 points) or two triangles sharing an vertex (not a pentagon), and check whether all $\binom{4}{2}$ or $\binom{5}{2}$ distances between our 2 shapes match up. If they do, we replace or extend the permutations we are keeping track of, and disregard any permutations from before which don't satisfy the conditions of this bullet-point.
\item If our current and previous triangles $\textit{don't share points}$, we essentially repeat steps 2-5 and $\textit{extend}$ the permutations we are keeping track of, in a similar manner to that shown in step 6.
\item At this point we have traversed through all triangles in both $\mathcal{A}_1$ and $\mathcal{B}_1$ with the same area, and have constructed permutations (not necessarily all of the same size) which can be considered as $\textit{sub-correspondence}$ of points between $P$ and $Q$ (meaning that more points may be included to the correspondences). We are now going to be considering the triangles with area of the next lowest mode and repeat steps 2-8, while keeping track of the permutations we have thus far. Some steps though will be slightly modified as now we are considering various shapes (corresponding to our permutations), and in the above steps when referring to our "$\textit{previous triangle}$", we will now be considering our "$\textit{previous shapes}$".
\item Repeating the above until we traverse through all triangles in $\mathcal{A}_1$ and $\mathcal{B}_1$ will give us a certain number of permutations, and for our problem we can simply take the permutations of  the largest size (might have multiple) and the points which aren't included in that permutation can be considered as $\textit{bad points}$ for the problem. Note that certain points might be considered as $\textit{bad}$ for certain permutations and not for others, which depends entirely on $P$ and $Q$.
\end{enumerate}

\subsection*{Brief Explanation On The Above Approach}
The idea of the above approach is to disregard non-identical shapes and configurations of the point sets $P$ and $Q$, while simultaneously constructing the desired permutations of $\textit{sub-configurations}$ which have the same shape. Note that we start of with the triangle areas which have the smallest mode in order to simplify the implementation of this algorithm. There will exist a diffeomorphism between $P$ and $Q$ $\textit{if and only if}$ the maximum permutations constructed have size $n$, where all points will be included. A drawback of this approach, is that we keep track of a relatively large number of permutations through out this process, but when going through each set of triangles of the same area, a lot of them are disregarded in step 7.


\subsection{Considering Areas Of Quadrilaterals}
Alternatively, we can partition $P=\{p_1,...,p_n\}$ and $Q=\{q_1,...,q_n\}$, by partitioning them into a total of $\binom{n}{4} \textit{ quadrilaterals }$, and consider the $\binom{4}{2}=6$ distances between our 4 points in each case. If we take 4 distinct points indexed $i,j,k,l$, we have the set of distances $\mathcal{DP}_{ijkl}=\{dp_{ij},dp_{ik},dp_{il},dp_{jk},dp_{jl},dp_{kl}\}$ and analogously $\mathcal{DQ}_{i'j'k'l'}=\{dq_{i'j'},dq_{i'k'},dq_{i'l'},dq_{j'k'},dq_{j'l'},dq_{k'l'}\}$.\\
\\
We now compute the areas as follows:
$$r:=\text{Diagonal}\{\mathcal{DP}_{ijkl} \} \qquad s:=\text{Diagonal}\{\mathcal{DP}_{ijkl}\backslash \{r\}\}$$
$$\text{\textit{r} and \textit{s} correspond to the diagonals of the quadrilateral.}$$
$$\{a,b,c,d\}:=\mathcal{DP}_{ijkl}\backslash\{r,s\}$$
$$\text{where \textit{a,c} correspond to distances of edges which don't share a vertex}$$
$$A_{ijkl} = \frac{1}{4}\sqrt{4r^2s^2-(a^2+c^2-b^2-d^2)^2}$$
\hrule
$$r':=\text{Diagonal}\{\mathcal{DQ}_{i'j'k'l'} \} \qquad s':=\text{Diagonal}\{\mathcal{DQ}_{i'j'k'l'}\backslash \{r\}\}$$
$$\text{\textit{r'} and \textit{s'} correspond to the diagonals of the quadrilateral.}$$
$$\{a',b',c',d'\}:=\mathcal{DQ}_{i'j'k'l'}\backslash\{r',s'\}$$
$$\text{where \textit{a',c'} correspond to distances of edges which don't share a vertex}$$
$$B_{i'j'k'l'} = \frac{1}{4}\sqrt{4r'^2s'^2-(a'^2+c'^2-b'^2-d'^2)^2}$$

and consider the sets of areas
$$\mathcal{A}=\big\{A_{ijkl} |  \forall \{i,j,k,l\}\subseteq\{1,...,n\}\big\}$$
$$\mathcal{B}=\big\{B_{i'j'k'l'} |  \forall \{i',j',k',l'\}\subseteq\{1,...,n\}\big\}$$
where $|\mathcal{A}|=|\mathcal{B}|=\binom{n}{4}=\frac{n(n-1)(n-2)(n-3)}{24}$.
\vspace{5mm}

We further partition the above sets as follows:
$$\mathcal{A}_1 = \big\{A_{ijkl} | A_{ijkl}\in \mathcal{A} \text{ and } \exists B_{i'j'k'l'} \in \mathcal{B} \text{ s.t. } A_{ijkl}=B_{i'j'k'l'} \big\}$$
$$\mathcal{B}_1 = \big\{B_{i'j'k'l'} | B_{i'j'k'l'}\in \mathcal{B} \text{ and } \exists A_{ijkl} \in \mathcal{A} \text{ s.t. } B_{i'j'k'l'}=A_{ijkl} \big\}$$
$$\mathcal{A}_2 = \mathcal{A} \backslash \mathcal{A}_1 \qquad \qquad \mathcal{B}_2 = \mathcal{B} \backslash \mathcal{B}_1$$

We can now follow the same algorithm described in section 3.1, with the exception that now we'll be considering 4 points at a time, rather than 3. Depending on the point-configurations $P$ and $Q$, either this approach or the previous approach might be more efficient, but this cannot be determined \textit{a priori}.


\section{Partition Into Polygons For $\varepsilon$-distortions}

We extend our previous work to $\varepsilon$-distortions.
\subsection{Areas Of Triangles For $\varepsilon$-distortions}
For notational convenience we will be using the same notation used in section 3.1, as well as the fact that our sets will have the following property:
$$(1-\varepsilon_{ij}) \leq \frac{dp_{ij}}{dq_{i'j'}} = \frac{||p_i-p_j||}{||q_i'-q_j'||} \leq (1+\varepsilon_{ij}) \text{, } \forall \{i,j\} \subseteq \{1,...,n\} \text{, given } i\neq j$$
rather than $dp_{ij} = dq_{i'j'} \Leftrightarrow ||p_i-p_j||=||q_i'-q_j'||$.
\vspace{5mm}

\begin{Thm}
  \textit{For our usual setup, it holds that for three points in our two point configurations $P$ and $Q$ with indices and areas $\{i,j,k\}$, $A_{ijk}$ and $\{i',j',k'\}$, $B_{i'j'k'}$ respectively, the points can be mapped from $P$ to $Q$ through an $E$-distorted diffeomorphism if and only if}
  $$\sqrt{(B_{i'j'k'})^2-\frac{1}{16}\cdot H_{1}} \leq A_{ijk} \leq \sqrt{(B_{i'j'k'})^2+\frac{1}{16}\cdot H_{2}}$$
  \textit{where $H_1,H_2$ depend on} $E := max\big\{\varepsilon_{st}|\{s,t\} \subseteq \{i,j,k\}\big\}$, \textit{and the elements of the distribution of distances of} $B_{i'j'k'}$.\\
\end{Thm}

\begin{proof}
  Considering the area of the triangles defined by the point $p_i,p_j,p_k$ and the corresponding points $q_i',q_j',q_k'$, we define $\varepsilon_{ij-} := (1-\varepsilon_{ij})$, $\varepsilon_{ij+} := (1+\varepsilon_{ij})$, and get the following 3 inequalities for each triangle:
  $$dq_{i'j'} \cdot \varepsilon_{ij-} \leq dp_{ij} \leq dq_{i'j'} \cdot \varepsilon_{ij+}$$
  $$dq_{i'k'} \cdot \varepsilon_{ik-} \leq dp_{ik} \leq dq_{i'k'} \cdot \varepsilon_{ik+}$$
  $$dq_{j'k'} \cdot \varepsilon_{jk-} \leq dp_{jk} \leq dq_{j'k'} \cdot \varepsilon_{jk+}$$

  In order to simplify our computations we define:
  $$E := max\big\{\varepsilon_{st}|\{s,t\} \subseteq \{i,j,k\}\big\}$$
  $$E_{-}:=(1-E) \qquad E_{+}:=(1+E)$$
  $$\Longrightarrow dq_{s't'} \cdot E_{-} \leq dp_{st} \leq dq_{s't'} \cdot E_{+} \text{ , for all pairs \{s,t\}} \subseteq \{i,j,k\}$$
  and
  $$s := dp_{ij}+dp_{ik}+dp_{jk}$$
  $$s' := dq_{i'j'}+dq_{i'k'}+dq_{j'k'}$$
  \vspace{2mm}

  It then follows that for all pairs $\{s,t\} \subseteq \{i,j,k\}$, that
  $$(2dq_{s't'}) \cdot E_{-} \leq 2dp_{st} \leq (2dq_{s't'}) \cdot E_{+}$$
  $$(-2dq_{s't'}) \cdot E_{+} \leq -2dp_{st} \leq (-2dq_{s't'}) \cdot E_{-}$$
  and
  $$(dq_{i'j'}+dq_{i'k'}+dq_{j'k'}) \cdot E_{-} \leq dp_{ij}+dp_{ik}+dp_{jk} \leq (dq_{i'j'}+dq_{i'k'}+dq_{j'k'}) \cdot E_{+}$$
  $$s' \cdot E_{-} \leq s \leq s' \cdot E_{+}$$
  $$\Longrightarrow (s'\cdot E_{-} -2dq_{i'j'}\cdot E_{+}) \leq (s-2dq_{ij}) \leq (s'\cdot E_{+} -2dq_{i'j'}\cdot E_{-})$$
  \vspace{2mm}

  Taking advantage of the $\textit{triangle inequality}$, $s'\geq 2dq_{i'j'}$, we get the following $\textit{bounds}$:
  $$2dq_{i'j'}\cdot(E_{-} - E_{+}) \leq (s'\cdot E_{-} -2dq_{i'j'}\cdot E_{+}) \leq (s-2dq_{ij}) \leq (s'\cdot E_{+} -2dq_{i'j'}\cdot E_{-}) \leq 2s'\cdot(E_{+} - E_{-})$$
  $$2dq_{i'j'}\cdot(E_{-} - E_{+}) \leq  (s-2dq_{ij}) \leq 2s'\cdot(E_{+} - E_{-})$$
  $$(-4E)\cdot dq_{i'j'} \leq  (s-2dq_{ij}) \leq (4E)\cdot s'$$
  $$0 \leq  (s-2dq_{ij}) \leq (4E)\cdot s'$$
  \vspace{2mm}
  We know that the area of the triangle defined by the points in the configurations $P$ and $Q$ are respectively:
  $$A_{ijk} = \sqrt{\frac{s(\frac{s}{2}-dp_{ij})(\frac{s}{2}-dp_{ik})(\frac{s}{2}-dp_{jk})}{2}} = \frac{1}{4}\cdot\sqrt{s(s-2dp_{ij})(s-2dp_{ik})(s-2dp_{jk})}$$
  $$\Longrightarrow A_{ijk} = \frac{1}{4}\cdot\sqrt{S} \text{ for } S:= s(s-2dp_{ij})(s-2dp_{ik})(s-2dp_{jk})$$
  
  $$B_{i'j'k'} = \sqrt{\frac{s'(\frac{s'}{2}-dq_{i'j'})(\frac{s'}{2}-dq_{i'k'})(\frac{s'}{2}-dq_{j'k'})}{2}} = \frac{1}{4}\cdot\sqrt{s'(s'-2dq_{i'j'})(s'-2dq_{i'k'})(s'-2dq_{j'k'})}$$
  $$\Longrightarrow B_{i'j'k'} = \frac{1}{4}\cdot\sqrt{S'} \text{ for } S':= s'(s'-2dq_{i'j'})(s'-2dq_{i'k'})(s'-2dq_{j'k'})$$

  In order to be as precise as possible we don't undertake any simplifications, and  from the above inequalities considering the indices $\{i,j,k\}$ and $\{i',j',k'\}$, we get:
  $$s'\prod_{\iota'\neq\kappa'}(s'\cdot E_{-}-2dq_{\iota'\kappa'}\cdot E_{+}) \leq  s\prod_{\iota\neq\kappa}(s-2dp_{\iota\kappa}) \leq s'\prod_{\iota'\neq\kappa'}(s'\cdot E_{+}-2dq_{\iota'\kappa'}\cdot E_{-})$$
  \hrule
  $$s'\prod_{\iota'\neq\kappa'}[(s'-2dq_{\iota'\kappa'})-(s'+2dq_{\iota'\kappa'})\cdot E] \leq  s\prod_{\iota\neq\kappa}(s-2dp_{\iota\kappa}) \leq s'\prod_{\iota'\neq\kappa'}[(s'-2dq_{\iota'\kappa'})+(s'+2dq_{\iota'\kappa'})\cdot E]$$
  \hrule
  $$s'(\alpha_1-\beta_1)(\alpha_2-\beta_2)(\alpha_3-\beta_3) \leq s\prod_{\iota\neq\kappa}(s-2dp_{\iota\kappa}) \leq s'(\alpha_1+\beta_1)(\alpha_2+\beta_2)(\alpha_3+\beta_3)$$
  \hrule
  $$s'\Big[\alpha_1\alpha_2\alpha_3-[\alpha_3\beta_2(\alpha_1-\beta_1)+\alpha_1\beta_3(\alpha_2-\beta_2)+\alpha_2\beta_1(\alpha_3-\beta_3)]-\beta_1\beta_2\beta_3\Big] \leq s\prod_{\iota\neq\kappa}(s-2dp_{\iota\kappa}) \leq$$
  $$\leq s'\Big[\alpha_1\alpha_2\alpha_3+[\alpha_3\beta_2(\alpha_1+\beta_1)+\alpha_1\beta_3(\alpha_2+\beta_2)+\alpha_2\beta_1(\alpha_3+\beta_3)]+\beta_1\beta_2\beta_3\Big]$$
  \hrule
  $$S'-s'\Big[\alpha_3\beta_2(\alpha_1-\beta_1)+\alpha_1\beta_3(\alpha_2-\beta_2)+\alpha_2\beta_1(\alpha_3-\beta_3)+\beta_1\beta_2\beta_3\Big] \leq S \leq$$
  $$\leq S'+s'\Big[\alpha_3\beta_2(\alpha_1+\beta_1)+\alpha_1\beta_3(\alpha_2+\beta_2)+\alpha_2\beta_1(\alpha_3+\beta_3)+\beta_1\beta_2\beta_3\Big]$$
  \hrule
  $$S'-H_{1} \leq S \leq S'+H_{2}$$

  Comparing the areas of two corresponding triangles from the 2 point-configurations we then get:
  $$16\cdot (B_{i'j'k'})^2-H_{1} \leq 16\cdot (A_{ijk})^2 \leq 16\cdot (B_{i'j'k'})^2+H_{2} $$
  $$(B_{i'j'k'})^2-\frac{1}{16}\cdot H_{1} \leq (A_{ijk})^2 \leq (B_{i'j'k'})^2+\frac{1}{16}\cdot H_{2} $$
  $$\sqrt{(B_{i'j'k'})^2-\frac{1}{16}\cdot H_{1}} \leq A_{ijk} \leq \sqrt{(B_{i'j'k'})^2+\frac{1}{16}\cdot H_{2}}$$

\end{proof}

\subsection{Considering Areas Of Triangles - Part 2}
For areas of triangles for $\varepsilon$-distortions, we construct the sets of areas of the partitioned triangles as follows:
$$\mathcal{A}=\{A_{ijk} |  \forall \{i,j,k\}\subseteq\{1,...,n\}\}$$
$$\mathcal{B}=\{B_{i'j'k'} |  \forall \{i',j',k'\}\subseteq\{1,...,n\}\}$$

$$\mathcal{A}_1 = \big\{A_{ijk} | A_{ijk}\in \mathcal{A} \text{ w/ \textit{E} and } \exists B_{i'j'k'} \in \mathcal{B} \text{, s.t. } |\sqrt{(B_{i'j'k'})^2-\frac{H_{1}}{16}}| \leq A_{ijk} \leq |\sqrt{(B_{i'j'k'})^2+\frac{H_{2}}{16}}|\big\}$$
$$\mathcal{B}_1 = \big\{B_{i'j'k'} | B_{i'j'k'}\in \mathcal{B} \text{ and } \exists A_{ijk} \in \mathcal{A} \text{ w/ \textit{E}, s.t. } |\sqrt{(B_{i'j'k'})^2-\frac{H_{1}}{16}}| \leq A_{ijk} \leq |\sqrt{(B_{i'j'k'})^2+\frac{H_{2}}{16}}|\big\}$$

$$\mathcal{A}_2 = \mathcal{A} \backslash \mathcal{A}_1 \qquad \qquad \mathcal{B}_2 = \mathcal{B} \backslash \mathcal{B}_1$$

We then follow the exact same $\textit{10-step algorithm}$ to get the desired result for  $\varepsilon-$distortions, although now it is very unlikely that 2 or more triangles will have the exact same area.

\subsection{Areas Of Quadrilaterals For $\varepsilon$-distortions}
Just as above, for notational convenience we will be using the same notation used in section 3.1, as well as the fact that our sets will have the following property:
$$(1-\varepsilon_{ij}) \leq \frac{dp_{ij}}{dq_{i'j'}} = \frac{||p_i-p_j||}{||q_i'-q_j'||} \leq (1+\varepsilon_{ij}) \text{, } \forall \{i,j\} \subseteq \{1,...,n\}$$
rather than $dp_{ij} = dq_{i'j'} \Leftrightarrow ||p_i-p_j||=||q_i'-q_j'||$.
\vspace{5mm}

\begin{Thm}
  \textit{For our usual setup, it holds that for four points in our two point configurations $P$ and $Q$ with indices and areas $\{i,j,k,l\}$, $A_{ijkl}$ and $\{i',j',k',l'\}$, $B_{i'j'k'l'}$ respectively, the points can be mapped from $P$ to $Q$ through an $E$-distorted diffeomorphism if and only if}
  $$\sqrt{(B_{i'j'k'l'})^2\cdot(1+E^2)^2-\frac{\hat{H}_{2}}{16}} \leq A_{ijkl} \leq \sqrt{(B_{i'j'k'l'})^2\cdot(1+E^2)^2+\frac{\hat{H}_{2}}{16}}$$
  \textit{where $\hat{H}_{1},\hat{H}_{2}$ depend on $E := max\big\{\varepsilon_{st}|\{s,t\} \subseteq \{i,j,k,l\}\big\}$, \textit{and the elements of the distribution of distances of} $B_{i'j'k'l'}$}.

\end{Thm}
\vspace{2mm}

\begin{proof}
  We consider our two n-point configurations $P=\{p_1,...,p_n\}$ and $Q=\{q_1,...,q_n\}$, and partition them into a total of $\binom{n}{4} \textit{ quadrilaterals }$, and take into account $\binom{4}{2}=6$ distances between our 4 points in each case. If we take the 4 points indexed $i,j,k,l$, we have the set of distances $\mathcal{DP}_{ijkl}=\{dp_{ij},dp_{ik},dp_{il},dp_{jk},dp_{jl},dp_{kl}\}$ and analogously $\mathcal{DQ}_{i'j'k'l'}=\{dq_{i'j'},dq_{i'k'},dq_{i'l'},dq_{j'k'},dq_{j'l'},dq_{k'l'}\}$ for our 2nd configuration. We also define $\varepsilon_{ij-} := (1-\varepsilon_{ij})$, $\varepsilon_{ij+} := (1+\varepsilon_{ij})$, and get the following 6 inequalities for each triangle:
  $$dq_{i'j'} \cdot \varepsilon_{ij-} \leq dp_{ij} \leq dq_{i'j'} \cdot \varepsilon_{ij+} \qquad \qquad dq_{j'k'} \cdot \varepsilon_{jk-} \leq dp_{jk} \leq dq_{j'k'} \cdot \varepsilon_{jk+}$$
  $$dq_{i'k'} \cdot \varepsilon_{ik-} \leq dp_{ik} \leq dq_{i'k'} \cdot \varepsilon_{ik+} \qquad \qquad dq_{j'l'} \cdot \varepsilon_{jl-} \leq dp_{jl} \leq dq_{j'l'} \cdot \varepsilon_{jl+}$$
  $$dq_{i'l'} \cdot \varepsilon_{il-} \leq dp_{il} \leq dq_{i'l'} \cdot \varepsilon_{il+} \qquad \qquad dq_{k'l'} \cdot \varepsilon_{kl-} \leq dp_{kl} \leq dq_{k'l'} \cdot \varepsilon_{kl+}$$

  Following a similar approach to what was shown previously, we define the following parameters and compute the areas:
  $$E := max\big\{\varepsilon_{st}|\{s,t\} \subseteq \{i,j,k,l\}\big\}$$
  $$E_{-}:=(1-E) \qquad E_{+}:=(1+E)$$
  $$\Longrightarrow dq_{s't'} \cdot E_{-} \leq dp_{st} \leq dq_{s't'} \cdot E_{+} \text{ , for all pairs \{s,t\}} \subseteq \{i,j,k,l\}$$
  and
  $$r:=\text{Diagonal}\{\mathcal{DP}_{ijkl} \} \qquad s:=\text{Diagonal}\{\mathcal{DP}_{ijkl}\backslash \{r\}\}$$
  $$\{a,b,c,d\}:=\mathcal{DP}_{ijkl}\backslash\{r,s\} \text{, where \textit{a,c} correspond to distances of edges which don't share a vertex}$$
  $$S:=(a^2+c^2-b^2-d^2) \qquad \tilde{S}:=(a^2+b^2+c^2+d^2)$$
  $$A_{ijkl} = \frac{1}{4}\sqrt{4r^2s^2-(a^2+c^2-b^2-d^2)^2} \Longrightarrow A_{ijkl} = \frac{1}{4}\sqrt{4r^2s^2-S^2}$$
  \hrule
  $$r':=\text{Diagonal}\{\mathcal{DQ}_{i'j'k'l'} \} \qquad s':=\text{Diagonal}\{\mathcal{DQ}_{i'j'k'l'}\backslash \{r'\}\}$$
  $$\{a',b',c',d'\}:=\mathcal{DQ}_{i'j'k'l'}\backslash\{r',s'\} \text{, where \textit{a',c'} correspond to distances of edges which don't share a vertex}$$
  $$S':=(a'^2+c'^2-b'^2-d'^2) \qquad \tilde{S'}:=(a'^2+b'^2+c'^2+d'^2)$$
  $$B_{i'j'k'l'} = \frac{1}{4}\sqrt{4r'^2s'^2-(a'^2+c'^2-b'^2-d'^2)^2} \Longrightarrow B_{i'j'k'l'} = \frac{1}{4}\sqrt{4r'^2s'^2-S'^2}$$
  \vspace{2mm}

  It then follows that for all pairs $\{s,t\} \subseteq \{i,j,k\}$
  $$(dq_{s't'})^2 \cdot (E_{-})^2 \leq (dp_{st})^2 \leq (dq_{s't'})^2 \cdot (E_{+})^2$$
  $$-(dq_{s't'})^2 \cdot (E_{+})^2 \leq -(dp_{st})^2 \leq -(dq_{s't'})^2 \cdot (E_{-})^2$$
  which imply that

  $$(r's')^2\cdot (E_{-})^4 \leq (rs)^2 \leq (r's')^2\cdot (E_{+})^4$$
  and
  $$\Big[(a'^2+c'^2)\cdot (E_{-})^2-(b'^2+d'^2)\cdot (E_{+})^2\Big] \leq (a^2+c^2-b^2-d^2) \leq \Big[(a'^2+c'^2)\cdot (E_{+})^2-(b'^2+d'^2)\cdot (E_{-})^2\Big]$$
  \hrule
  $$\Big[(a'^2+c'^2)\cdot (1-2E+E^2)-(b'^2+d'^2)\cdot (1+2E+E^2)\Big] \leq (a^2+c^2-b^2-d^2) \leq$$
  $$\leq \Big[(a'^2+c'^2)\cdot (1+2E+E^2)-(b'^2+d'^2)\cdot (1-2E+E^2)\Big]$$
  \hrule
  $$\Big[(a'^2+c'^2-b'^2-d'^2)\cdot(1+E^2)-2E\cdot(a'^2+c'^2+b'^2+d'^2)\Big] \leq (a^2+c^2-b^2-d^2) \leq$$
  $$\leq \Big[(a'^2+c'^2-b'^2-d'^2)\cdot(1+E^2)+2E\cdot(a'^2+c'^2+b'^2+d'^2)\Big]$$
  \hrule
  $$\Big[S'\cdot(1+E^2)-\tilde{S'}\cdot(2E)\Big] \leq S \leq \Big[S'\cdot(1+E^2)+\tilde{S'}\cdot(2E)\Big]$$
  \hrule
  $$-\Big[S'\cdot(1+E^2)+\tilde{S'}\cdot(2E)\Big]^2 \leq -S^2 \leq -\Big[S'\cdot(1+E^2)-\tilde{S'}\cdot(2E)\Big]^2$$
  \hrule
  $$-S'^2\cdot(1+E^2)^2-\Big[\tilde{S'}\cdot(2E)\cdot[2E\tilde{S'}+S'(1+E^2)]\Big] \leq -S^2 \leq$$
  $$\leq -S'^2\cdot(1+E^2)^2-\Big[\tilde{S'}\cdot(2E)\cdot[2E\tilde{S'}-S'(1+E^2)]\Big]$$
  \hrule
  $$-S'^2\cdot(1+E^2)^2-H_1 \leq -S^2 \leq -S'^2\cdot(1+E^2)^2-H_2$$

  Comparing the areas of two corresponding quadrilaterals from the 2 point-configurations we then get:
  $$4(r's')^2\cdot(E_{-})^4-S'^2\cdot(1+E^2)^2-H_1 \leq 4(rs)^2-S^2 \leq 4(r's')^2\cdot(E_{+})^4 -S'^2\cdot(1+E^2)^2-H_2$$
  \hrule
  $$4(r's')^2\cdot[(1+E^2)^2-4E(1-E+E^2)]-S'^2\cdot(1+E^2)^2-H_1 \leq 4(rs)^2-S^2 $$
  $$\leq 4(r's')^2\cdot[(1+E^2)^2+4E(1+E+E^2)] -S'^2\cdot(1+E^2)^2-H_2$$
  \hrule
  $$\Big[4(r's')^2-S'^2\Big]\cdot(1+E^2)^2-\Big[16(r's')^2\cdot(E-E^2+E^3)+H_1\Big] \leq 4(rs)^2-S^2 \leq$$
  $$\leq \Big[4(r's')^2-S'^2\Big]\cdot(1+E^2)^2+\Big[16(r's')^2\cdot(E+E^2+E^3)-H_2\Big]$$
  \hrule
  $$\Big[4(r's')^2-S'^2\Big]\cdot(1+E^2)^2-\hat{H}_{1} \leq 4(rs)^2-S^2 \leq \Big[4(r's')^2-S'^2\Big]\cdot(1+E^2)^2+\hat{H}_{2}$$
  \hrule
  $$(B_{i'j'k'l'})^2\cdot(1+E^2)^2-\frac{\hat{H}_{1}}{16} \leq (A_{ijkl})^2 \leq B_{i'j'k'l'}^2\cdot(1+E^2)^2+\frac{\hat{H}_{2}}{16}$$

  \hrule
  $$\sqrt{(B_{i'j'k'l'})^2\cdot(1+E^2)^2-\frac{\hat{H}_{1}}{16}} \leq A_{ijkl} \leq \sqrt{(B_{i'j'k'l'})^2\cdot(1+E^2)^2+\frac{\hat{H}_{2}}{16}}$$

\end{proof}

\subsection{Considering Areas Of Quadrilaterals - Part 2}
For areas of quadrilaterals for $\varepsilon$-distortions, we construct the sets of areas of the partitioned quadrilaterals as follows:

$$\mathcal{A}=\big\{A_{ijkl} |  \forall \{i,j,k,l\}\subseteq\{1,...,n\}\big\}$$
$$\mathcal{B}=\big\{B_{i'j'k'l'} |  \forall \{i',j',k',l
'\}\subseteq\{1,...,n\}\big\}$$

$$\mathcal{A}_1 = \big\{ A_{ijkl} | A_{ijkl}\in \mathcal{A} \text{ w/ \textit{E} and } \exists B_{i'j'k'l'} \in \mathcal{B} \text{, s.t. } |\sqrt{(B_{i'j'k'l'})^2\cdot(1+E^2)^2-\frac{\hat{H}_{1}}{16}}|$$
$$\leq A_{ijkl} \leq |\sqrt{(B_{i'j'k'l'})^2\cdot(1+E^2)^2+\frac{\hat{H}_{2}}{16}}|\big\}$$

$$\mathcal{B}_1 = \big\{B_{i'j'k'l'} | B_{i'j'k'l'}\in \mathcal{B} \text{ and } \exists A_{ijkl} \in \mathcal{A} \text{ w/ \textit{E}, s.t. } |\sqrt{(B_{i'j'k'l'})^2\cdot(1+E^2)^2-\frac{\hat{H}_{1}}{16}}|\leq$$
$$\leq A_{ijkl} \leq |\sqrt{(B_{i'j'k'l'})^2\cdot(1+E^2)^2+\frac{\hat{H}_{2}}{16}}|\big\}$$

$$\mathcal{A}_2 = \mathcal{A} \backslash \mathcal{A}_1 \qquad \qquad \mathcal{B}_2 = \mathcal{B} \backslash \mathcal{B}_1$$

We then follow the exact same $\textit{10-step algorithm}$ to get the desired result for  $\varepsilon-$distortions, although now it is very unlikely that 2 or more quadrilaterals will have the exact same area.

\section{Reconstruction From Distances}

\subsection{One-Sided Error Algorithm}

We want to see how likely it is to $\textit{Construct Point Configurations}$, given the distance distributions. See [8] for more on measuring the distance between 2D point sets.\\

Using [3] Theorem 1.3 and considering $n$-\textit{point} configurations, we can select $\binom{n}{3}$ different sets for $\{i_0,i_1,i_2\}$, $\binom{n-3}{2}$  for $\{j_1,j_2\}$, $\binom{n-5}{2}$  for $\{k_1,k_2\}$, $\binom{n-7}{2}$  for $\{l_1,l_2\}$ and $\binom{n-9}{2}$  for $\{m_1,m_2\}$.

The total number of possible such $collections$ is:
\begin{align*}
  \binom{n-3}{2}\cdot\prod_{\iota=1}^4\binom{n-3-2\iota}{2} &= \frac{n(n-1)(n-2)}{3!}\cdot\prod_{\iota=1}^4\frac{(n-2\iota-1)(n-2\iota-2)}{2!}\\
  &=\frac{n!}{(n-11)!}\cdot\frac{1}{96}
\end{align*}
Define $N:=\frac{n!}{(n-11)!}\cdot\frac{1}{96}$.\\

\begin{figure}[h]
  \centerline{\includegraphics[width=16cm]{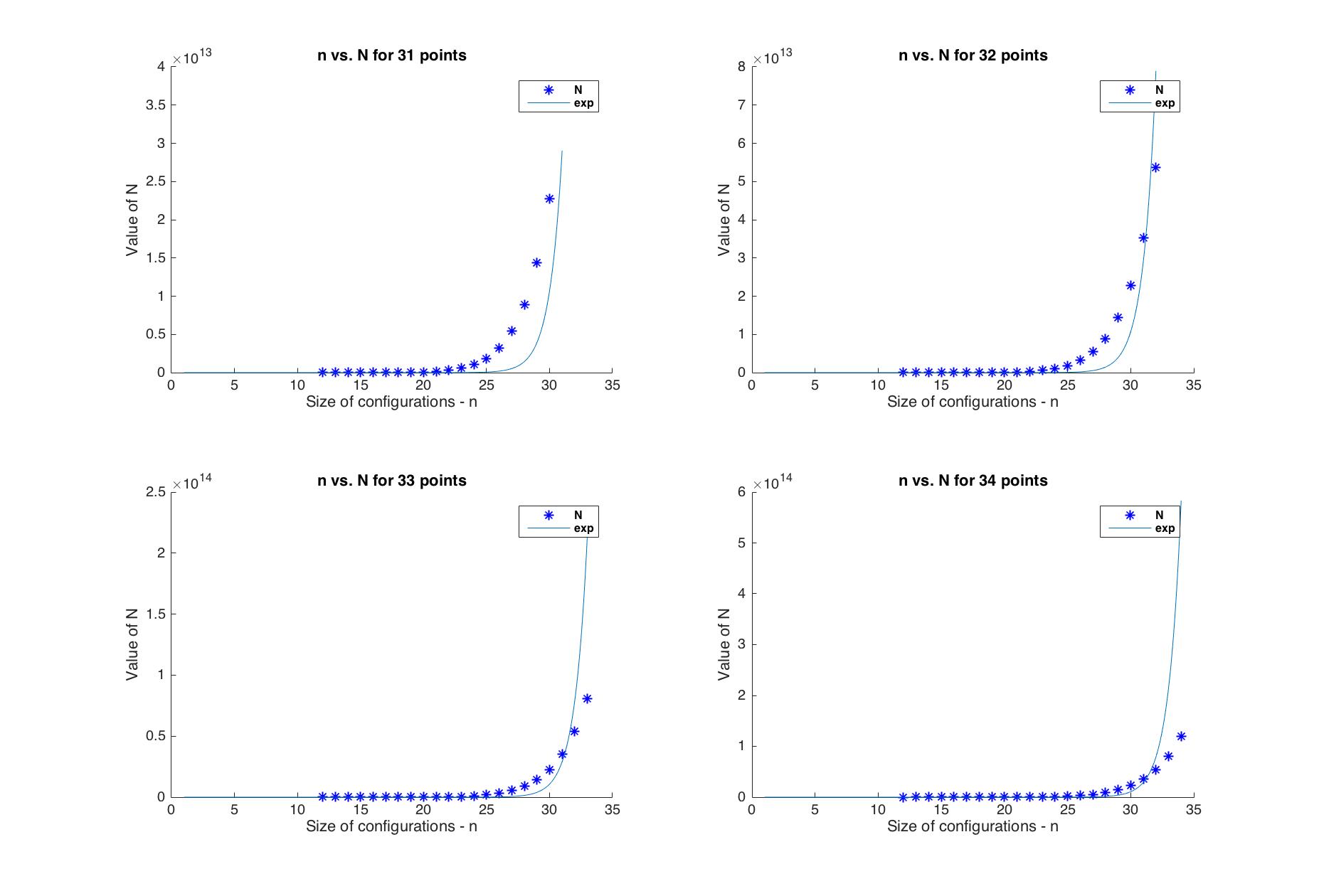}}
  \caption{$N$ tends to behave like $e^n$, as $n$ approaches 33}
\end{figure}

As shown above, the number of 11-tuples one has to check is very large and not practical even though it's relatively easy to implement.
An alternative way to check if our configuration is in fact reconstructible from distances is to run a one sided-error algorithm.

\pagebreak
We make use of the polynomial defined in [3] which takes as inputs 6 distances:

$g(U,V,W,X,Y,Z) := \Big[ 2U^2Z+2UZ^2+2V^2Y+2VY^2+2X^2W+2XW^2+2UVX+2UYW+2VWZ+2XYZ\Big]
+\Big[-2UVY-2UVZ-2UXW-2UXZ-2UYZ-2UWZ-2VXY-2VXW-2VYW-2VYZ-2XYW-2XWZ\Big]$\\
and the following facts:
\begin{itemize}
\item $g(U,V,W,X,Y,Z)=0$ \textit{if and only if} the inputs are the sides and diagonals of a well-defined quadrilateral
\item if $g(d_{\{i_0,i_1\}},d_{\{i_0,i_2\}},d_{\{j_1,j_2\}},d_{\{k_1,k_2\}},d_{\{l_1,l_2\}},d_{\{m_1,m_2\}})\neq0$ for all such 11-tuples in $P$, then $P$ is reconstructible from distances
\end{itemize}
\vspace{5mm}

For the one-sided error algorithm we \textit{assume} that $P$ \textbf{is reconstructible} from distances, and select at random an 11-tuple from $P$. If for the given 11-tuple we get that $g(d_{\{i_0,i_1\}},...,d_{\{m_1,m_2\}})=0$, we conclude that $P$ is in fact \textbf{not reconstructible from distances}. The issue here is that we may falsely conclude that $P$ is reconstructible from distances, with an error of $\frac{|\mathcal{K}_1|}{N}$, where:
$$\mathcal{K} = \Big\{\{i_0,i_1,...,m_2\}|\{i_0,i_1,...,m_2\}\subseteq P\Big\} \qquad \qquad |\mathcal{K}|=N$$
$$\mathcal{K}_1 = \Big\{\{i_0,i_1,...,m_2\}|g(d_{\{i_0,i_1\}},d_{\{i_0,i_2\}},d_{\{j_1,j_2\}},d_{\{k_1,k_2\}},d_{\{l_1,l_2\}},d_{\{m_1,m_2\}}) \neq 0 \Big\}$$
$$\mathcal{K}_2 = \Big\{\{i_0,i_1,...,m_2\}|g(d_{\{i_0,i_1\}},d_{\{i_0,i_2\}},d_{\{j_1,j_2\}},d_{\{k_1,k_2\}},d_{\{l_1,l_2\}},d_{\{m_1,m_2\}}) = 0 \Big\}$$\\

This obvious depends purely on $P$ and the cardinality of $\mathcal{K}_2$, and there's nothing that can be said about it a priori. So depending on how many such \textit{bad 11-tuples} exist in $P$, we will either have a large or a small error. In order to reduce this we can randomly select $x$ such tuples and check all of if them, and if at least one satisfies
$g(d_{\{i_0,i_1\}},...,d_{\{m_1,m_2\}})=0$, we safely conclude that $P$ is not reconstructible from distances. In the case where all the tuples we selected lie in $\mathcal{K}_1$ we will have a false conclusion, where the error will be $\Big(\frac{|\mathcal{K}_1|}{N}\Big)^x\ll\frac{|\mathcal{K}_1|}{N}$.

\subsection{Generalizing The Results Of [3] For $\varepsilon$-distortions}

\begin{Thm}
  \textit{For a generic} $P,Q \subset \R^2 $, \textit{where the following conditions hold:}
  \begin{itemize}
  \item \textit{if} $dist(P)=dist(Q)$, \textit{where} $dist(R)=\{dist(r_i,r_j)|r_i,r_j\in R\subset \R^2\}$
  \item \textit{all} $\binom{n}{2}$ \textit{distances} $dist(p_i,p_j)$ \textit{are distinct and} $g(U,V,W,X,Y,Z)=0$
  \item \textit{if} $\{e_1,...,e_6\}\in E$ \textit{for} $E:=\{(p_i,p_j)|i,j\in\{i,...,n\}\}$, \textit{are not the diagonals of a quadrilateral
    then} $g(e_1,...,e_6)\neq0$
  \end{itemize}
  \textit{then} $P\cong Q$.\\
\end{Thm}

\begin{Def}
  \textit{By} $dist(P)\approx dist(Q)$, \textit{we mean that for each element in} $dist(P)$ \textit{there exists only one element in} $dist(Q)$, \textit{such that} $(1-\varepsilon)\leq \frac{dp_{ij}}{dq_{i'j'}}\leq (1+\varepsilon)$ \textit{and vice versa}.
\end{Def}

For our case, we generalize the theorem as follows:\\

\begin{Thm}
  \textit{For a generic $P,Q \subset \R^2 $, if $\exists$ $T \in A(D)$ such that $|T(P)-Q|<\varepsilon$, given $\varepsilon>0$, then the following inequalities hold:}
  \begin{itemize}

  \item \textit{if} $dist(P)\approx dist(Q)$, \textit{where} $dist(R)=\{dist(r_i,r_j)|r_i,r_j\in R\subset \R^2\}$.
  \item \textit{all} $\binom{n}{2}$ \textit{distances} $dist(p_i,p_j)$ \textit{are distinct and} $$\big[g(U',...,Z')\cdot(1+3\varepsilon^2) - H\big] \leq g(U,...,Z) \leq \big[g(U',...,Z')\cdot(1+3\varepsilon^2) + H\big]$$
    \textit{where $H$ depends on $\varepsilon$, the polynomial $g(\cdot)$, and the distances $\{U',...,Z'\}$}.
  \item \textit{if $\{e_1,...,e_6\}\in E$ for $E:=\{(p_i,p_j)|i,j\in\{i,...,n\}\}$ are not the diagonals of a quadrilateral, then} $$ g(e_1,...,e_6) \notin \bigg[\Big[ \big(g'_1\cdot(1-\varepsilon)^3 + g'_2\cdot(1+\varepsilon)^3\Big],\Big[ \big(g'_1\cdot(1+\varepsilon)^3 + g'_2\cdot(1-\varepsilon)^3\Big]\bigg].$$
  \end{itemize}
\end{Thm}

\begin{proof}
  We show the derivation of the analogous conditions for the $\varepsilon$-distortions

  \begin{itemize}
  \item The first bullet-point is essentially what we want for the $\varepsilon$-distortions.

  \item We have:
    \begin{align*}
      g(U,V,W,X,Y,Z) &:= \Big[ 2U^2Z+2UZ^2+2V^2Y+2VY^2+2X^2W+2XW^2\\
        &+2UVX+2UYW+2VWZ+2XYZ\Big]\\
      &+\Big[-2UVY-2UVZ-2UXW-2UXZ-2UYZ-2UWZ\\
        &-2VXY-2VXW-2VYW-2VYZ-2XYW-2XWZ\Big]\\
      &=\Big[g_2\Big]+\Big[g_1\Big]
    \end{align*}
    and we know that select our 6-\textit{distance} collections, in order to satisfy $$\forall \alpha \in \{U,V,W,X,Y,Z\}\subseteq dist(P) \text{, } \exists \alpha' \in \{U',V',W',X',Y',Z'\}\subseteq dist(Q),$$ $$\text{s.t. } (1-\varepsilon)\cdot \alpha' \leq \ \alpha \leq (1+\varepsilon)\cdot \alpha'$$
    W.L.O.G., we label and reorder our $6-distance$ collections in the following way:
    $$\{U,V,X,X,Y,Z\} \mapsto \{\alpha_1,\alpha_2,\alpha_3,\alpha_4,\alpha_5,\alpha_6\}\subseteq dist(P)$$
    $$\{U',V',X',X',Y',Z'\} \mapsto \{\alpha_1',\alpha_2',\alpha_3',\alpha_4',\alpha_5',\alpha_6'\}\subseteq dist(Q)$$
    We then get the following inequalities:
    $$\forall I\subset\{\alpha_1,...,\alpha_6\} \text{ s.t. } \Big(2\prod_{i\in I}{a_{i}}\Big) \text{ is a term of } g(\cdot), \text{ and the corresponding \textit{I'} to \textit{I}} $$
    $$\Big(2\prod_{i'\in I'}{a'_{i'}}\Big)\cdot(1-\varepsilon)^3 \leq \Big(2\prod_{i\in I}{a_{i}}\Big) \leq \Big(2\prod_{i'\in I'}{a'_{i'}}\Big)\cdot(1+\varepsilon)^3$$
    $$-\Big(2\prod_{i'\in I'}{a'_{i'}}\Big)\cdot(1+\varepsilon)^3 \leq -\Big(2\prod_{i\in I}{a_{i}}\Big) \leq -\Big(2\prod_{i'\in I'}{a'_{i'}}\Big)\cdot(1-\varepsilon)^3$$

    $$ \Longrightarrow g'_2\cdot(1+\varepsilon)^3 \leq g_2 \leq g'_2\cdot(1-\varepsilon)^3$$

    $$ \Longrightarrow \big[g'_1\cdot(1-\varepsilon)^3 + g'_2\cdot(1+\varepsilon)^3\big] \leq g(U,...,Z) \leq \big[g'_1\cdot(1+\varepsilon)^3 + g'_2\cdot(1-\varepsilon)^3\big]$$
    \hrule
    $$ \big[g(U',...,Z')\cdot(1+3\varepsilon^2) + (g'_2-g'_1)\cdot(3\varepsilon+\varepsilon^3)\big] \leq g(U,...,Z) \leq$$
    $$\leq \big[g(U',...,Z')\cdot(1+3\varepsilon^2) + (g'_1-g'_2)\cdot(3\varepsilon+\varepsilon^3)\big]$$
    \hrule
    $$ g(U',...,Z')\cdot(1+3\varepsilon^2) - H \leq g(U,...,Z) \leq g(U',...,Z')\cdot(1+3\varepsilon^2) + H$$
    $$\text{where } H=(g'_1-g'_2)\cdot(3\varepsilon+\varepsilon^3)$$

  \item Follow the same steps shown in the proof of the second inequality, with the only difference that the inequality signs are switched, since we want $\textit{not almost equality}$. So if in the proof above we had $\beta_1 \leq \alpha \leq \beta_2$, we would now use : $\alpha < \beta_1 \text{ or } \alpha > \beta_2$ $\Leftrightarrow \alpha \notin [\beta_1,\beta_2].$

  \end{itemize}
\end{proof}


\subsection{Construction Of Points Given Distance Distribution}

\textbf{Question :}
\textit{Given $dist(P)$ and the total volume $V$ of the convex set of points $P\subset \R^3$, with $P$ unique, how do we reconstruct $P\subset\R^3$ upto a rigid motion?}\\

In order to solve this problem, we can take the following steps, which directly relate to the $\textit{10-step algorithm}$:
\begin{enumerate}
\item List the distances in increasing order, and call this ordered list $D$, with $E:=|D|=\binom{n}{2}$. All $\binom{n}{2}$ edges correspond to a side of $\binom{n}{3}$ triangles of the $P$ configuration.
\item Find all possible triangles for $d_i\in D$ in increasing order, with $d_i$ being the smallest edge of the triangle. So essentially:

  \fbox{\begin{minipage}{20em}
      \begin{algorithm}[H]
        $i\gets1$\;
        $h\gets1$\;
        \While{$i\leq (E-3)$}{
          take all $\binom{n-i}{2}$ ordered 3-tuples $\{d_i,d_{j>i},d_{E\geq k>j}\}$\;
          \eIf{$(d_k-d_j)\geq d_i$}{
            $\mathcal{T}_h^{(i)}\gets\{d_i,d_j,d_k\}$\;
            $\mathcal{T}^{(i)} = \{\mathcal{T}^{(i)}, \mathcal{T}_h^{(i)}\}$\;
            $h\gets h+1$\;
          }{
            disregard 3-tuple $\{d_i,d_j,d_k\}$\;
          }
          $i\gets i+1$\;
        }
        $\mathcal{T}^{(i)} \gets \bigcup_{i=1}^{E-3} \mathcal{T}^{(i)}$\;
        $T \gets |\mathcal{T}^{(i)}|$\;
      \end{algorithm}
    \end{minipage}}

    So $\mathcal{T}$ would look like $\mathcal{T}=\big\{\{d_1,d_2,d_3\},\{d_1,d_4,d_5\},\{d_2,d_4,d_5\},...\big\}$ for instance, which is still ordered in ascending order of the first element of the 3-tuple, then the second element and then the third element.

  \item Again in ascending order of $d_i\in D$ take all triangle 3-tuples $\mathcal{T}_{\iota}\in\mathcal{T}$ for $\iota\in\{1,...,T\}$, and then construct all possible 6-tuples of edges which form a tetrahedron with $d_i$ being the minimum length of its edges, by essentially combining the triangles. This can be done as follows:

    \fbox{\begin{minipage}{38em}
        \begin{algorithm}[H]
          take $\mathcal{T}^{(i)}$ for all $i\in\{1,...,E-3\}$\ from above\;
          $i\gets1; \qquad \qquad$ \% index of set $\mathcal{T}^{(i)}$, consisting of all triangles with $d_i$ as its smallest side\\
          $j\gets2; \qquad \qquad$ \% index of $\mathcal{T}^{(j)}$, for $E>j>i$\\

          \While{$i\leq (E-3)$}{
            $h\gets1; \qquad \qquad$ \% index for all elements of $\mathcal{T}^{(i)}$\\
            $\delta\gets1; \qquad \qquad$ \% index of 6-tuple tetrahedron considering $\mathcal{T}^{(i)}$\\
            $d'_1 \gets 1^{st}$ element of $\mathcal{T}_h^{(i)}$;  $h_1 \gets d'_1$'s original index\;
            $d'_2 \gets 2^{nd}$ element of $\mathcal{T}_h^{(i)}$;  $h_2 \gets d'_2$'s original index\;
            $d'_3 \gets 3^{rd}$ element of $\mathcal{T}_h^{(i)}$;  $h_3 \gets d'_3$'s original index\;
            $h \gets h+1$\;

            \While{$h\leq |\mathcal{T}^{(i)}|$}{
              $\tilde{d'}_1 \gets 1^{st}$ element of $\mathcal{T}_h^{(j)}$;  $h_4 \gets \tilde{d}'_1$'s original index\;
              $\tilde{d'}_2 \gets 2^{nd}$ element of $\mathcal{T}_h^{(j)}$;  $h_5 \gets \tilde{d}'_2$'s original index\;
              $\tilde{d'}_3 \gets 3^{rd}$ element of $\mathcal{T}_h^{(j)}$;  $h_6 \gets \tilde{d}'_3$'s original index\;
              \If {$\{d_{h_2},d_{h_3}\}\cap\{d_{h_4},d_{h_5}\} = \varnothing$}{
                find the triangles which have one of the following combinations of sides ($\textit{both can't occur simultaneously}$, and we only need to search in $\bigcup_{\eta\in \hat{H}}\mathcal{T}^{\eta}$ for $H:=\{h_2,h_3,h_4h_5\}$, $\hat{H}:=H\backslash max\{H\}$), and their corresponding $3^{rd}$ side is assigned below to $x_{\iota\in\{1,2\}}$:\\
                $\qquad \text{(i)} \{d_{h_2},d_{h_4},x_1\}$ \& $\{d_{h_3},d_{h_5},x_2\}$ \\
                $\qquad \text{(ii)} \{d_{h_2},d_{h_5},x_1\}$ \& $\{d_{h_3},d_{h_4},x_2\}$ \\
                \eIf{$x_1 = x_2$}{
                  $h_6 \gets x_1$\;
                  $\Delta^{(i)}_{\delta} \gets \{d_{h_1},d_{h_2},d_{h_3},d_{h_4},d_{h_5},d_{h_6}\}$\;
                  $\Delta^{(i)} \gets \{\Delta^{(i)},\Delta^{(i)}_{\delta}\}$\;
                  $\mathcal{V}^{(i)}_{\delta} \gets |\frac{1}{288}\sqrt{det\begin{pmatrix}
                    0 & 1 & 1 & 1 & 1  \\
                    1 & 0 & (d_{h_2})^2 & (d_{h_3})^2 & (d_{h_6})^2\\
                    1 & (d_{h_2})^2 & 0 & (d_{h_1})^2 & (d_{h_4})^2\\
                    1 & (d_{h_3})^2 & (d_{h_1})^2 & 0 & (d_{h_5})^2\\
                    1 & (d_{h_6})^2 & (d_{h_4})^2 & (d_{h_5})^2 & 0 \end{pmatrix}}|$\;

                  $\mathcal{V}^{(i)} \gets \{\mathcal{V}^{(i)},\mathcal{V}^{(i)}_{\delta} \}$\;
                  $\delta\gets\delta+1$\;
                }{
                  disregard current 6-tuple;
                }
              }
              $h \gets h+1$\;
            }{
            }
            $i\gets i+1$\;
          }
          $\mathcal{V} \gets \bigcup_{\iota=1}^{E-3} \mathcal{V}^{(\iota)}; \qquad \qquad$ \% Set with all tetrahedron 6-tuples\\
          $\Delta \gets \bigcup_{\iota=1}^{E-3} \Delta^{(\iota)}; \qquad \qquad$ \% Set with the corresponding volumes of the  tetrahedrons\\
        \end{algorithm}
      \end{minipage}}

\item Now we want to find the collection of tetrahedrons which have a sum of volume equal to $V$. In order to do this, we re-arrange out sets $\mathcal{V}$ and $\Delta$ to $\hat{\mathcal{V}}$ and $\hat{\Delta}$, in ascending order with respect to the mode of the $\mathit{faces}$ of the tetrahedrons. We do this because it will be a lot faster to identify whether a tetrahedron is not part of the overall n-point configurations, and we will disregard it. We then go through the following steps:

    \begin{enumerate}
    \item Take in order the tetrahedron $\delta_i\in\hat{\Delta}$ and it's corresponding volume $\nu_i\in\hat{\mathcal{V}}$, and then it's 3-tuple "$\textit{triangle face}$" with the smallest mode, $t^{(i)}_1$.
    \item Find the next tetrahedron in the ordered set $\hat{\Delta}$ which has $t^{(i)}_1$ as a face, and combine the two potential tetrahedrons, to get a hexahedron and its volume.
    \item Considering the 6 3-tuple faces of the hexahedron we repeat step (b) and this is done until we exceed the total volume an n-tuple or we have an n-tuple with a total volume less than V.
    \item In this case, we take out the last tetrahedron which was added and add the next possible candidate and check the conditions from step (c) and repeat until all possible candidate tetrahedrons have been checked. If we are not successful, we then go "2 steps back" and take the next possible candidate for the our $2^{nd}$ most recent selection of a tetrahedron.
    \item We then repeat steps (b)-(d) considering the appropriate shape we have after each iteration.
    \item If we are not successful and have back-traced back to $\delta_i$, we consider the next appropriate tetrahedron from $\hat{\Delta}$ and repeat steps (b)-(e).
    \item Given the conditions on our set $dist(P)$, this algorithm will terminate once it successfully finds the tuple set of the exterior faces  of the overall convex shape formed by the configuration $P$, and the set of tetrahedrons $\Delta_{final}$ which form it.
    \end{enumerate}
  \item At this point, we know the triangles on the exterior of the overall convex shape we are looking for, as well as the which distances from our initial set $dist(P)$ correspond to which point. From here it is therefore only a matter of selecting an arbitrary tetrahedron from $\Delta_{final}$ and placing it in $\R^3$, and based on this initialization we construct 3 points $P$, consider the tetrahedrons in $\Delta_{final}$ adjacent to the one we initially constructed and then construct 4 more points, repeat this process until we have exhausted all elements of $\Delta_{final}$. This will therefore give us an orthogonal transformation of $P$.
\end{enumerate}

The construction of $P$ is unique upto a rigid motion, by the setup of the problem.\\
The point matching problem is further explored in [9].\\
\textbf{Open Question :}
\textit{Given $dist(P)$ for $P\subset\R^D$, how do we reconstruct $P$ upto a rigid motion?}

\subsection{Kabsch's Algorithm - Find Rotation And Translation [4],[5]}
There's a well known algorithm for finding the optimal rotation which minimizes the \textit{root mean squared deviation} between two paired sets of points. In our case there exists an exact rotation, so by simply running this algorithm we get the exact rotation. This algorithm is broken up into the following three steps:\\

\begin{itemize}
  \item{\textbf{Shift by Center of Mass}}
    $${CM_P}^{(j)} = \frac{1}{n}\sum_{i=1}^{n}{P_i}^{(j)} \qquad {CM_Q}^{(j)} = \frac{1}{n}\sum_{i=1}^{n}{\hat{Q}_i}^{(j)} \qquad j\in\{1,...,d\}$$
    $$\tilde{P}=(P-\vec{1} CM_P) \qquad \tilde{Q}=(Q-\vec{1} CM_Q) \qquad \text{for } \vec{1}\in\R^n \text{ and } \tilde{P},\tilde{Q} \in\R^{n\times d}$$

  \item{\textbf{Find the Optimal Rotation}}\\
    Here we use the \textit{singular value decomposition} of the \textit{covariance matrix} $H=\tilde{P}^T\tilde{Q}$.
    $$  H=U\Sigma V^T \qquad k=sign(det(VU^T))\cdot1 \qquad
    D = \begin{pmatrix} 1 & 0 & \cdots & 0 & 0\\
                        0 & 1 & \cdots & 0 & 0\\
                        \vdots & \vdots & \ddots &\vdots &\vdots\\
                        0 & 0 &\cdots & 1 & 0\\
                        0 & 0 &\cdots & 0 & k\\ \end{pmatrix}$$
    $$R = VDU^T$$
    The $D$ matrix is used to take care of any reflections which might have taken place.

  \item{\textbf{Find the Translation $\vec{t}$}}
  $$\vec{t} = (-R\times {CM_P}^T + {CM_Q}^T)^T \qquad \vec{t}\in\R^d$$

\end{itemize}

\subsection{Example Results And Visualization}

Below is an example where construct a random 60-point configuration $P$. Using $P$ we then construct $Q$, such that $P \sim Q$. We then use the method described in section 6.3 to label the points, and that in section 6.4 to confirm that the configurations align. Indeed, we got that $(P-P_{\pi}^{-1}QP_{\pi}) = 0$.



\begin{figure}[h]
\centerline{\includegraphics[width=25cm]{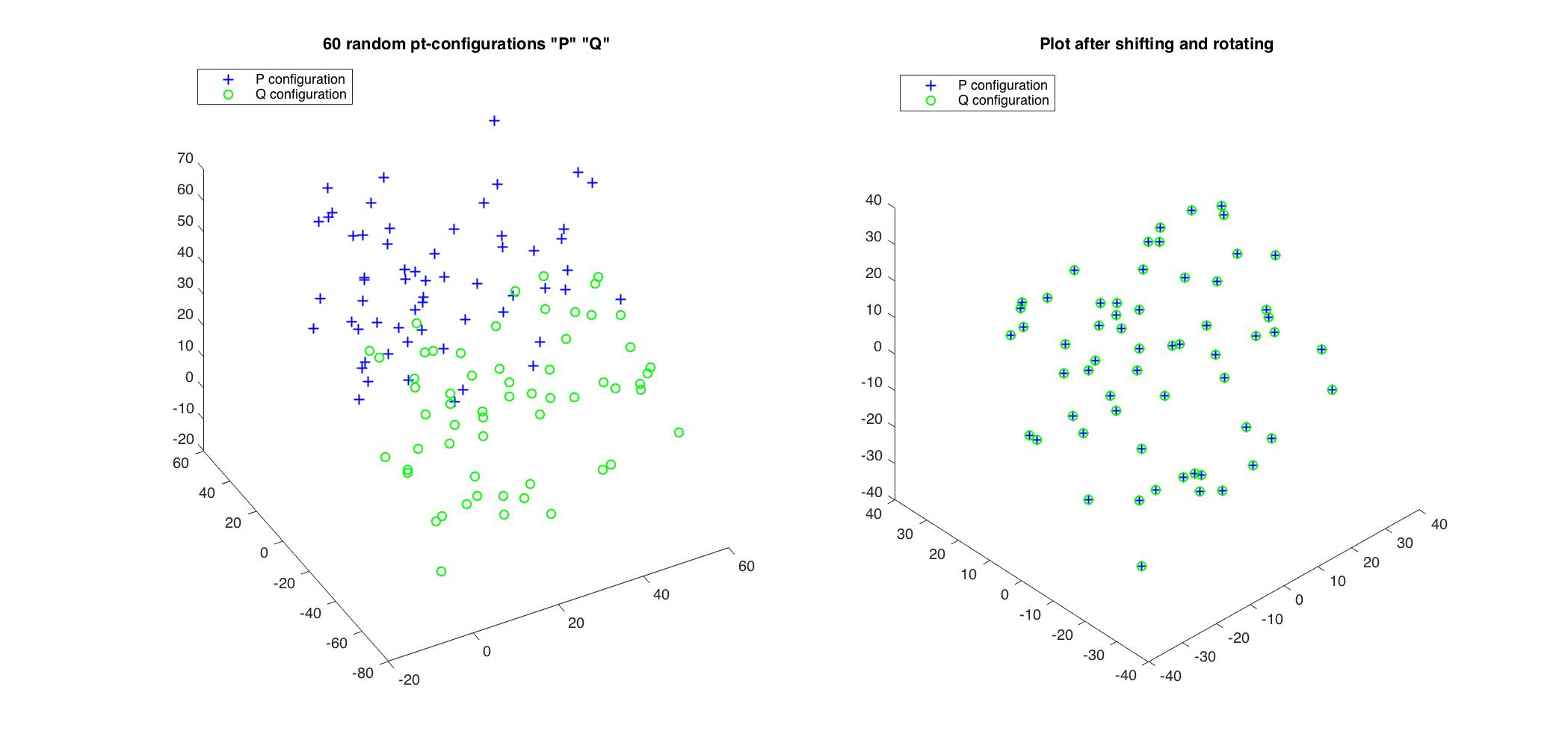}}
\caption{Plot with $P$ and $Q$ before and after alignment}
\end{figure}

\subsection{Alternative Way For Constructing Rotation After Relabelling And Shifting}
After shifting our two configurations as shown in section 6.4, we constructed the matrices $\tilde{P}$ and $\tilde{Q}$ with entries the coordinates of all points of $P$ and $Q$ with center of mass being the origin. Assuming that $n\geq D$, we can randomly select $D$ rows of $\tilde{P}$ and the corresponding rows in $\tilde{Q}$ (we have already relabelled them), and construct a \textit{change of basis}. In order to do this we need to first confirm that the vectors are \textit{linearly independent} which is pretty simple using software, as you simply need to make sure that the $D\times D$ matrices consisting of these vectors in their columns have a \textit{nonzero determinant}, or are \textit{full-rank}.
So if we have selected $D$ linearly independent vector from $\tilde{P}$, which we denote as $\mathfrak{B}_P=\{\vec{p_1},...,\vec{p_D}\}$, and their corresponding vectors in $\tilde{Q}$, $\mathfrak{B}_Q=\{\vec{q_1},...,\vec{q_D}\}$, these form a basis for $\R^d$, and the rotation we are looking for is the \textit{change of basis} between $\mathfrak{B}_P$ and $\mathfrak{B}_Q$. So we essentially do the following to construct the rotation $R$:
$$\mathcal{P}_{D}=   \begin{pmatrix}
    \vec{p_1} &    \hdots &    \vec{p_D}
    \end{pmatrix} \in \R^{D\times D}     \qquad
\mathcal{Q}_{D}=   \begin{pmatrix}
    \vec{q_1} &    \hdots &    \vec{q_D}
    \end{pmatrix} \in \R^{D\times D}$$
$$R\mathcal{P}_{D}=\mathcal{Q}_{D} \qquad \Longrightarrow \qquad R=\mathcal{Q}_{D}{\mathcal{P}_{D}}^{-1}$$

\subsection{SVD Approach After Relabelling}

An alternative way of computing the permutation $R$ is using the $SVD$ and the fact that the singular values of a matrix are unique. We are considering the case where $n\geq D$ and $U \in \R^{D\times D}, \Sigma\in \R^{D\times n}, V \in \R^{n\times n} $, so we can truncate the \textit{right singular vectors matrix} $V$, as the first $D$ right-singular vectors for $\mathcal{P}=\begin{pmatrix} \vec{p_1} & \hdots & \vec{p_n} \end{pmatrix} \in \R^{D\times n}$ and $\mathcal{Q}=\begin{pmatrix} \vec{q_1} & \hdots & \vec{q_n} \end{pmatrix} \in \R^{D\times n}$ match up.

Notationally we use $A_D=A(1:D,1:D)$ to denote the truncation of matrix $A$, by taking the submatrix consisting of the columns and rows 1 through $D$ of $A$.

For $\mathcal{P} = U_P \Sigma_P {V_P}^T$ and $\mathcal{Q} = U_Q \Sigma_Q {V_Q}^T$, we know that $\Sigma_P = \Sigma_Q$ and that $\exists R\in O(D)$ s.t. $R\mathcal{P}=\mathcal{Q}$. It then follows that for:
$$ V_{P_D}:=V_P(1:D,1:D) \qquad V_{Q_D}:=V_Q(1:D,1:D) \qquad \Sigma_D := \Sigma_P(1:D,1:D) = \Sigma_Q(1:D,1:D) $$
$$ R\mathcal{P} = \mathcal{Q} \qquad \Rightarrow \qquad R(U_P \Sigma_P {V_P}^T) = U_Q \Sigma_Q {V_Q}^T \qquad \Rightarrow \qquad (R U_{P}) \Sigma_D {V_{P_D}}^T = U_{Q} \Sigma_D {V_{Q_D}}^T $$

If the singular values of $\mathcal{P}$ and $\mathcal{Q}$ are all distinct, then the $SVD$ of the two matrices are unique. If they are not distinct, then $\exists U_{P}, U_{Q}$ and $M$ a permutation matrix s.t. $M U_{P_D} = U_{Q_D}$, which implies that $V_{P_D} = V_{Q_D}$. So for simplicity let's consider this \textit{singular value decomposition}, where $M=I_{D\times D}$. It then follows that for all point configurations which are congruent upto a rigid motion, there exist a \textit{singular value decomposition} and $R\in O(D)$, s.t.:
$$ R U_{P} = U_{Q} \qquad V_{R_D} = V_{Q_D} \qquad \Longrightarrow \qquad R= U_Q{U_P}^{T}
$$

\subsection{On Matching Point Configurations}
\begin{Lem}
  $[6]$ Let $\mathcal{P},\mathcal{Q} \in \R^{D\times n}$ as defined above. Then $\mathcal{P}^{T}\mathcal{P}=\mathcal{Q}^{T}\mathcal{Q}$ \textit{if and only if} $\exists A \in O(D)$ such that $A\mathcal{P}=\mathcal{Q}$.
\end{Lem}

The above Lemma implies that a necessary and sufficient condition for two configurations to be equivalent, is that their \textit{Gramian-matrices} are equal, after translating them such that their center of mass is at the origin.

\subsection{Kabsch's Algorithm On  $\varepsilon$-diffeomorphisms}

Kabsch's algorithm also finds the rotation in order to align point configuration $P$ with a point configuration $Q$, where $\forall i\in\{1,...,n\} \text{  } \exists i'\in\{1,...,n\}$ such that $||p_i-q_{i'}||<\frac{\varepsilon}{2}$, for $p_i\in P$, and $q_{i'}\in Q$, with a given error. Note that for $\epsilon<1$ and the property required for $\varepsilon$-diffeomorphisms $(1-\varepsilon) \leq \frac{||p_i-p_j||}{||q_{i'}-q_{j'}||} \leq (1+\varepsilon)$, we get that $\big| |p_i-p_j|-|q_{i'}-q_{j'}|\big|\leq \varepsilon$. Below we justify that the condition on the points $||p_i-q_{i'}||<\frac{\varepsilon}{2}$, satisfy the stated equivalent condition for $\varepsilon$-diffeomorphisms:

$$
\big| |p_i-q_{i'}|-|p_j-q_{j'}| \big| = \big| \varepsilon_i - \varepsilon_j \big| \leq \frac{\varepsilon}{2} \qquad \text{as } \varepsilon_i, \varepsilon_j \in \Big(-\frac{\varepsilon}{2},\frac{\varepsilon}{2}\Big)
  $$
\hrule
$$
|p_i-q_{i'}| < \varepsilon_i \qquad |p_j-q_{j'}| < \varepsilon_j \qquad \Rightarrow \qquad \big| |p_i-p_j|-|q_{i'}-q_{j'}| \big| < \big| \varepsilon_i + \varepsilon_j \big|
  $$
  $$
  \big| |p_i-p_j|-|q_{i'}-q_{j'}| \big| \leq \big| |p_i-p_j|+|q_{i'}-q_{j'}| \big|
  $$

  $$
  \Longrightarrow \big| |p_i-p_j|-|q_{i'}-q_{j'}| \big| \leq \big| |p_i-p_j|+|q_{i'}-q_{j'}| \big| \leq \big| \varepsilon_i + \varepsilon_j \big| \leq \varepsilon
$$
$$
  \Longrightarrow \big| |p_i-p_j|-|q_{i'}-q_{j'}| \big| \leq \varepsilon
$$
\hrule
\vspace{5mm}

Since we have \textit{exact bounds} on the difference $|p_i-q_{i'}|$, we can easily implement simulation to confirm that Kabsch's algorithm works on \textit{random} n-point configurations $P$ with points in $Q$ which satisfy the above condition. The alignment will obviously \textit{not be exact}. An alternate solution to the point cloud registration problem can be seen in [7], for both the rigid and non-rigid cases.

\begin{figure}[h]
\centerline{\includegraphics[width=10cm]{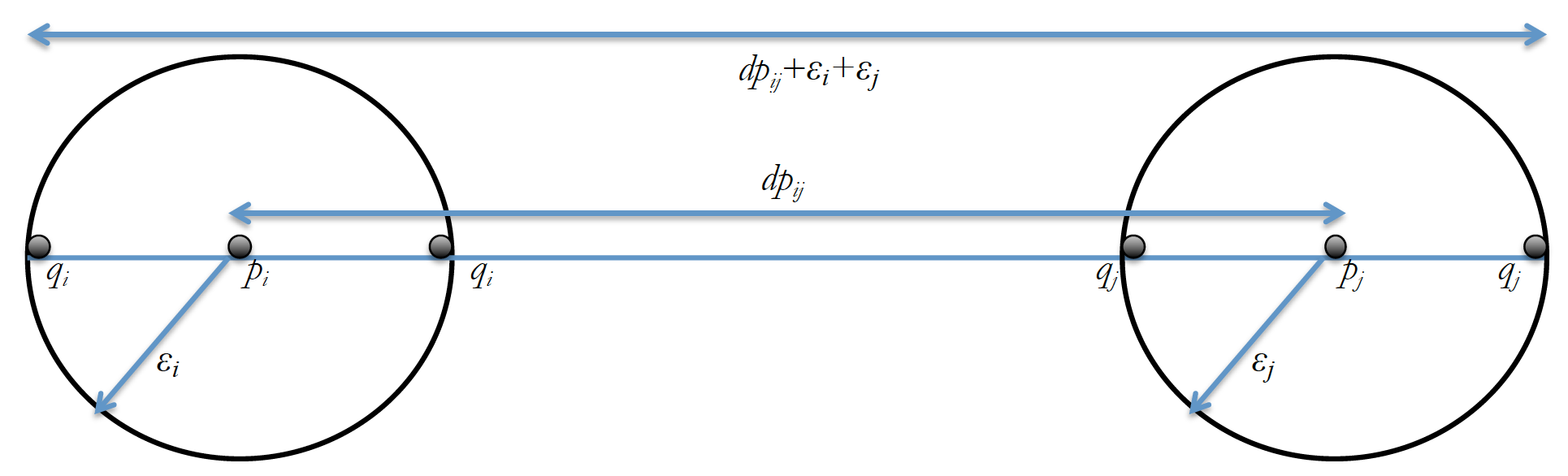}}
\caption{Depiction of inequality $\big| |p_i-p_j|-|q_{i'}-q_{j'}| \big| < \big| \varepsilon_i + \varepsilon_j \big|$}
\end{figure}

\begin{figure}[h]
\centerline{\includegraphics[width=15cm]{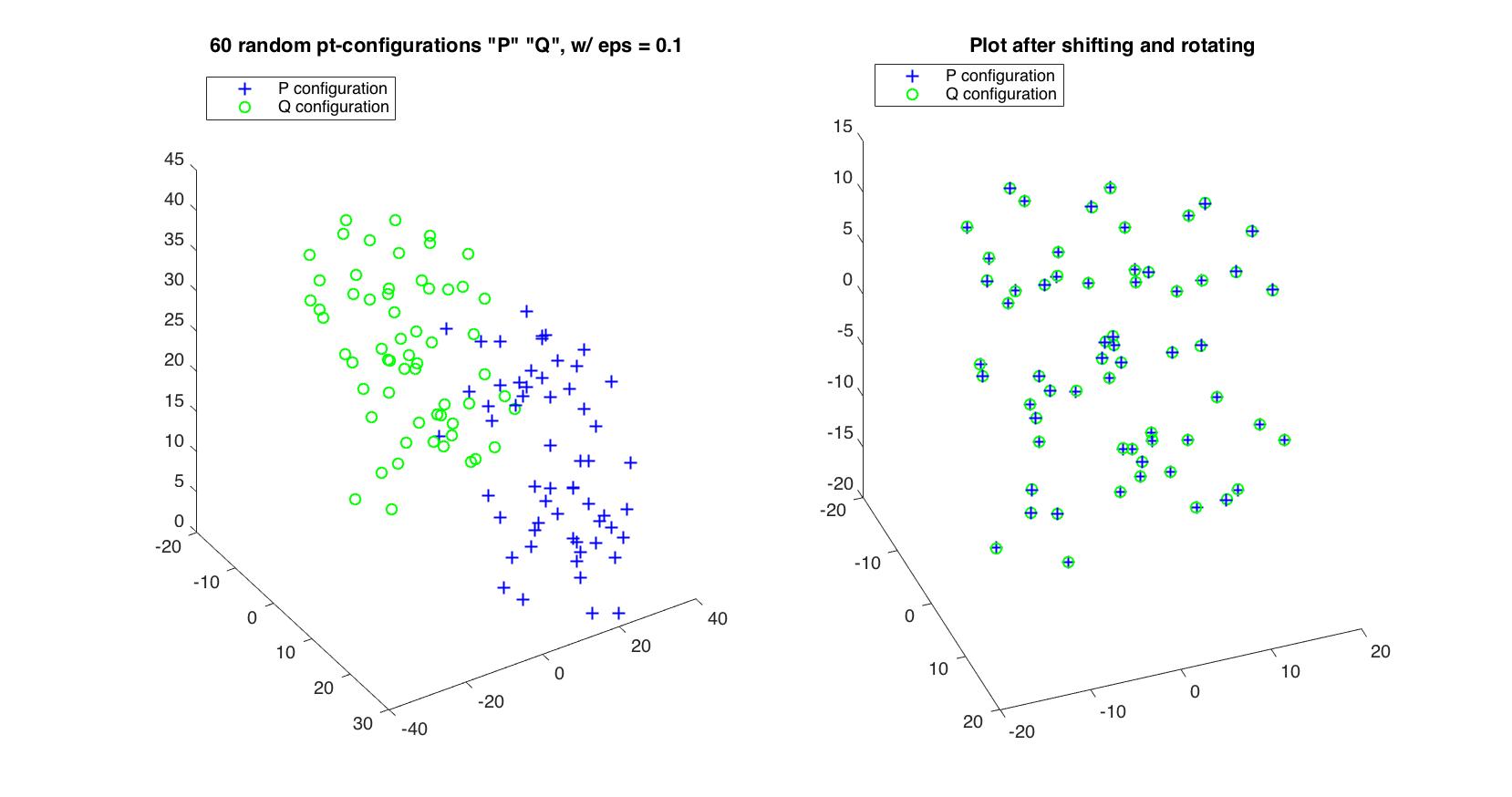}}
\caption{Plot with $P$ and $Q$ before and after alignment, for $\varepsilon$-diffeomorphisms}
\end{figure}

\subsection{Difference In Rotations On $\varepsilon$-diffeomorphisms Using Kabsch's Algorithm}

A simulation over multiple random $P$ and corresponding $Q$ configurations and their alignment was implemented, which resulted in the \textit{error-plots} provided on the next page. These were over $\varepsilon \in \{0.01,0.02,0.04,.06,0.08,0.1\}$ and $n\in\{10,12,12,...,150\}$. The error was calculated in means of \textit{sum of squared differences} for each coordinate over 30 averaged simulations, and then averaged the coordinates. What the plots reveal is simply that for greater $\varepsilon$ we have a greater averaged error for the same number of points, while there is no obvious trend. This is also justified by figure 10, where a \textit{Least-Squared Fit} was used, and it's obvious that as $n$ increases, the error also tends to increase.

An alternative, more detailed approach to such configurations, is described in [6].

\newpage


\begin{figure}[h]
\centerline{\includegraphics[width=18.8cm]{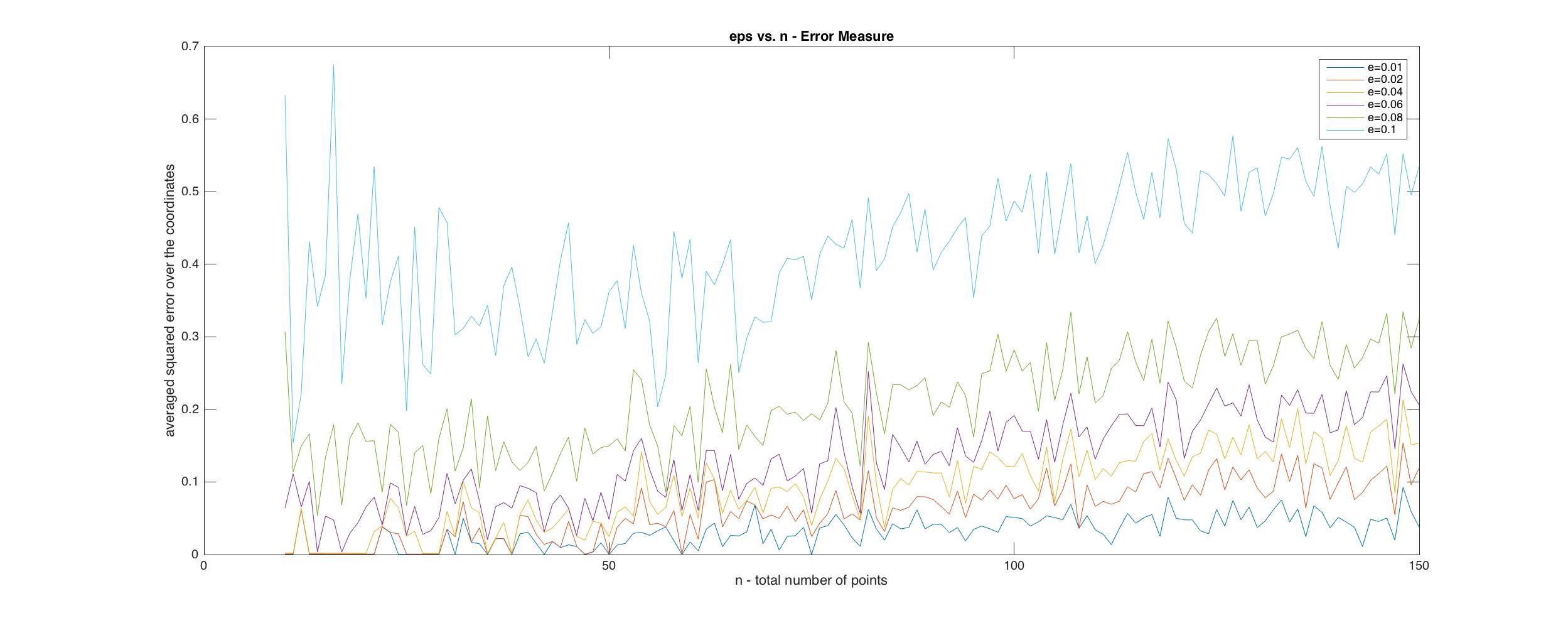}}
\caption{Error plot over 30 random samples, with values of $\varepsilon$ in the legend}
\end{figure}

\begin{figure}[h]
\centerline{\includegraphics[width=18.8cm]{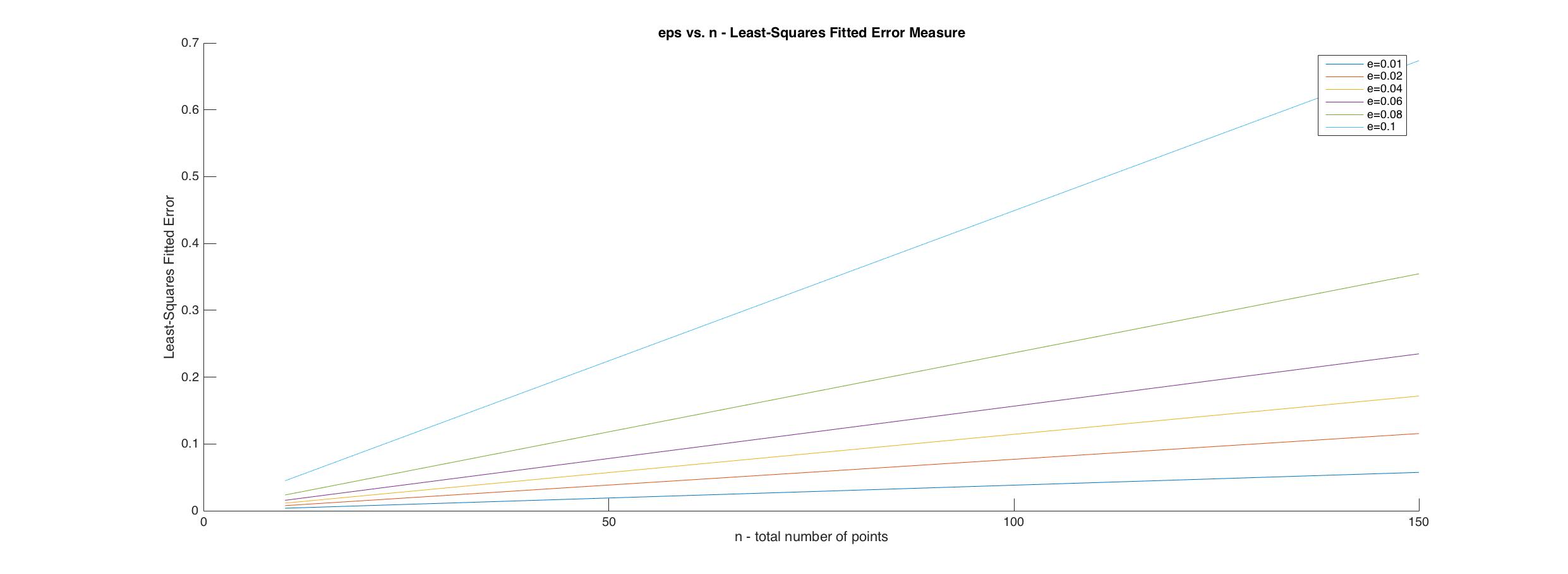}}
\caption{Linear Fit on error, using \textit{Least-Squares Fit}}
\end{figure}

\pagebreak

\vspace{-5mm}

\textbf{Acknowledgements:} Neophytos Charalambides was supported by the University of Michigan Math Department and Research Experience for Undergraduates program, Brad Schwartz was supported by the University of Michigan Undergraduate Research Opportunity Program, and Dr. Steven Damelin by the American Mathematical Society. Thanks goes to all departments and programs for the continued support.



\section*{References}

[0] S. B. Damelin, {\it On the Whitney extension problem for near isometries and beyond}, submitted for consideration for publication, arxiv:2103.09748.

[1] S.B.Damelin, C.Fefferman, \textit{On the Whitney Extension-Interpolation-Alignment problem for almost isometries with small distortion in $\R^D$}, arxiv:4011766

[2] M.Boutin, G.Kemper, \textit{On Reconstructing n-Point Configurations from the Distributions of Distances or Areas}, from https://arxiv.org/abs/math/0304192

[3] M.Boutin, G.Kemper, \textit{Which Point Configurations are Determined by the Distribution of their Pairwise Distances?}, from https://arxiv.org/abs/math/0311004

[4] Nghia Ho. (n.d.). Retrieved June 15, 2016, from http://nghiaho.com/?page$\_$id=671

[5] Kabsch algorithm. (n.d.). Retrieved June 15, 2016, from\\\indent https://en.wikipedia.org/wiki/Kabsch$\_$algorithm

[6] D.Jimenez, G.Petrova, \textit{On Matching Point Configurations}, from \\http://www.math.tamu.edu/~gpetrova/JP.pdf

[7]H. Maron, N. Dym, I. Kezurer, S. Kovalsky, Y. Lipman, \textit{Point Registration via Efficient Convex Relaxation}, from\\
https://services.math.duke.edu/\textasciitilde shaharko/projects/ProcrustesMatchingSDP\_lowres.pdf

[8] M. Werman, D. Weinshall, \textit{Similarity and Affine Distance Between Point Sets}, from\\
http://www.cs.huji.ac.il/\textasciitilde werman/Papers/comp.pdf

[9] E. Arkin, K. Kedem, J. Mitchell, J. Sprinzak, M. Werman, \textit{Matching Points into Pairwise-Disjoin Noise Regions: Combinatorial Bounds and Algorithms}, from\\
http://www.cs.huji.ac.il/\textasciitilde werman/Papers/match-point-regs.pdf

\end{document}